\documentclass[11pt]{amsart}

\usepackage{amsrefs, mathrsfs,amsmath, amscd, amsthm,amssymb, amsfonts, verbatim,subfigure, enumerate,stmaryrd, slashed}
\usepackage[mathscr]{euscript}
\usepackage{graphicx}
\usepackage[top=1in, bottom=1.25in, left=1.3in, right=1.3in]{geometry}

\usepackage{mathpazo}
\usepackage[colorlinks=true,linkcolor=blue,citecolor=red]{hyperref}

\usepackage{parskip}
\usepackage[all,cmtip]{xy}
\newcommand{\xym}{\xymatrix@1@=14pt@M=2pt}

\usepackage{tikz}
\usepackage{tikz-cd}
\usetikzlibrary{arrows,shapes}
\usetikzlibrary{trees}
\usetikzlibrary{matrix,arrows}
\usetikzlibrary{positioning}
\usetikzlibrary{calc,through}
\usetikzlibrary{decorations.pathreplacing}
\usepackage{pgffor}

\usetikzlibrary{decorations.pathmorphing}
\usetikzlibrary{decorations.markings}
\tikzset{
    vector/.style={decorate, decoration={snake}, draw},
	provector/.style={decorate, decoration={snake,amplitude=2.5pt}, draw},
	antivector/.style={decorate, decoration={snake,amplitude=-2.5pt}, draw},
    fermion/.style={draw=black, postaction={decorate},
        decoration={markings,mark=at position .55 with {\arrow[draw=black]{>}}}},
    fermionbar/.style={draw=black, postaction={decorate},
        decoration={markings,mark=at position .55 with {\arrow[draw=black]{<}}}},
    fermionnoarrow/.style={draw=black},
    gluon/.style={decorate, draw=black,
        decoration={coil,amplitude=4pt, segment length=5pt}},
    scalar/.style={dashed,draw=black, postaction={decorate},
        decoration={markings,mark=at position .55 with {\arrow[draw=black]{>}}}},
    scalarbar/.style={dashed,draw=black, postaction={decorate},
        dwecoration={markings,mark=at position .55 with {\arrow[draw=black]{<}}}},
    scalarnoarrow/.style={dashed,draw=black},
    electron/.style={draw=black, postaction={decorate},
        decoration={markings,mark=at position .55 with {\arrow[draw=black]{>}}}},
	bigvector/.style={decorate, decoration={snake,amplitude=4pt}, draw},
}

\newtheorem{theorem}{Theorem}[section]
\newtheorem{prop}[theorem]{Proposition}
\newtheorem{lemma}[theorem]{Lemma}
\newtheorem{cor}[theorem]{Corollary}

\theoremstyle{definition}
\newtheorem{dfn}[theorem]{Definition}
\newtheorem{dfn/lem}{Definition/Lemma}

\theoremstyle{remark}
\newtheorem{rmk}[theorem]{Remark}
\newtheorem{eg}[theorem]{Example}

\usepackage{color}   
\usepackage{hyperref}
\hypersetup{
    colorlinks=true, 
    linktoc=all,     
    linkcolor=blue,  
}

\newcommand{\on}{\operatorname}
\newcommand{\ol}{\overline}
\newcommand{\nc}{\newcommand}
\nc{\wt}{\widetilde}

\nc{\UU}{\mathbf{U}}
\newcommand{\R}{R}
\nc{\cR}{\mathcal{R}}
\nc{\C}{\mathbb{C}}

\newcommand{\Z}{\mathbb{Z}}

\nc{\U}{\mathbf{U}}
\nc{\rr}{\mathbf{r}}
\nc{\zz}{\mathbf{z}}
\nc{\PD}{\on{PD}}
\nc{\D}{\on{Discs}}

\newcommand{\mc}{\mathcal}

\nc{\kk}{\mathbf{k}}
\nc{\CEcoh}{\mathcal{C}^{*}}

\nc{\fh}{\mathfrak{h}}
\nc{\fg}{\mathfrak{g}}
\nc{\fghat}{\widehat{\fg}}
\nc{\fn}{\mathfrak{n}}
\nc{\cF}{\mc{F}}
\nc{\cG}{\mc{G}}
\nc{\CC}{\mathbb{C}}
\nc{\RR}{\mathbb{R}}
\nc{\Linf}{L_{\infty}}
\nc{\cL}{\mc{L}}

\nc{\ffg}{\fghat^{\#c}_{\pi}}
\nc{\ffga}{\fghat^{\#c, alg}_{\pi}}

\nc{\Ch}{\mc{C}^{Lie}_{*}}

\nc{\delbar}{\overline{\partial}}
\nc{\del}{\partial}
\nc{\dd}{d}

\nc{\vac}{|0\rangle}

\nc{\cK}{\mc{K}}
\nc{\opqm}[2]{\Omega^{#1,#2}_{m}}
\nc{\fgtil}{\wt{\fg}}
\nc{\Sym}{\on{Sym}}
\nc{\dzbar}{d \overline{z}}
\nc{\zbar}{\overline{z}}

\nc{\symcat}{\on{\bf C}^{\otimes}}
\nc{\Fcal}{\mathcal{F}}
\nc{\sF}{\mc{F}}
\nc{\sG}{\mc{G}}
\nc{\DVS}{\on{DVS}}

\nc{\Vect}{\on{Vect}}
\nc{\dgVect}{\on{dg-Vect}}

\nc{\ip}{\langle \bullet , \bullet \rangle}
\nc{\ses}[3]{0 \rightarrow #1 \rightarrow #2 \rightarrow #3 \rightarrow 0}

\DeclareMathOperator{\Ker}{Ker}

\def\d{{\rm d}}
\def\tensor{\otimes}
\def\Hat{\widehat}
\def\xto{\xrightarrow}

\allowdisplaybreaks[4]  

\begin{document}

 
 \title{Toroidal prefactorization algebras associated to holomorphic fibrations and a relationship to vertex algebras}
 
  \author{Matt Szczesny, Jackson Walters, and Brian Williams}
  \date{}

\begin{abstract}
Let $X$ be a complex manifold, $\pi: E \rightarrow X$ a locally trivial holomorphic fibration with fiber $F$, and $(\fg, \ip )$ a Lie algebra with an invariant symmetric form. We associate to this data a holomorphic prefactorization algebra  $\sF_{\fg, \pi}$ on $X$ in the formalism of Costello-Gwilliam. When $X=\mathbb{C}$, $\fg$ is simple, and $F$ is a smooth affine variety, we extract from $\sF_{\fg, \pi}$ a vertex algebra which is a vacuum module for the universal central extension of the Lie algebra $\fg \otimes H^{0}(F, \mc{O})[z,z^{-1}]$. As a special case, when $F$ is an algebraic torus $(\CC^{*})^n$, we obtain a vertex algebra naturally associated to an $(n+1)$--toroidal algebra, generalizing the affine vacuum module. 
\end{abstract}

\maketitle
\thispagestyle{empty}

\tableofcontents

\newpage

\section{Introduction}

Affine Kac-Moody algebras play an important role in many areas of mathematics and physics, specifically representation theory and conformal field theory. 
The affine Kac-Moody algebra $\Hat{\fg}$ associated to a complex Lie algebra $\fg$ is a central extension of the loop algebra $L \fg = \fg[z,z^{-1}]$. 
Toroidal algebras provide a generalization of affine Kac-Moody algebras to include ``multi-dimensional" loops.
The $(n+1)$-toroidal algebra is defined as a central extension of the Lie algebra
\[
L^{n+1} \fg = \fg[z, z_1^\pm \ldots, z_n^{\pm}] .
\] 
This is precisely the $n$th iterated algebraic loop space of the Lie algebra $\fg$. 
More generally, one can speak of a ``toroidal algebra" associated to any commutative $\CC$-algebra $A$ and Lie algebra $\fg$. 
This more general class of Lie algebras are defined as central extensions of Lie algebras of the form $A \tensor L \fg$, for which $A = \CC[z_1^{\pm}, \ldots, z_n^{\pm}]$ is the special case above. 

The vacuum module of the affine Kac-Moody Lie algebra has the structure of a {\em vertex algebra}, which turns out to be an extremely valuable tool when studying the representation theory of $\Hat{\fg}$. 
Connections between toroidal algebras and vertex algebras have been explored in the works \cite{MRY, MR, BBS, LTW}. 

Our goal in this paper is to use the language factorizations to provide an intrinsic geometric construction of vertex algebras associated to toroidal algebras. 
Being geometric, our construction exists in more generality than in the ordinary construction of toroidal algebras and their vertex algebra enhancements. 
Indeed, in addition to a Lie algebra $\fg$ with an invariant pairing, we start with any locally trivial holomorphic fibration $\pi: E \rightarrow X$ with fiber $F$.
Associated to this data, we produce a (pre)factorization algebra $\sF_{\fg, \pi}$ on $X$. 
Our model for factorization algebras follows the work of Costello-Gwilliam, which has been extensively developed in the work \cite{CG}. 

As will be explained below, $\sF_{\fg, \pi}$ can be viewed as a factorization envelope (see Section \ref{sec: prefact1}) with coefficients in the algebra $H^*(F, \mc{O}_F)$.  
When $X = \C$, and $F$ is a smooth complex affine variety, a result of Costello-Gwilliam extracts from $\sF_{\fg, \pi}$ (or more precisely, a certain dense subalgebra) a vertex algebra.
This vertex algebra can be described as an induced module for a central extension $\fghat_F$ of the Lie algebra
\begin{equation} \label{gf}
\fg \otimes H^{0}(F, \mc{O}_F)[z,z^{-1}]
\end{equation}
(here $H^{0}(F, \mc{O})$ denotes the $\C$-algebra of regular functions on $F$), that is universal when $\fg$ is simple.   When $F$ is a point and $\pi: E \mapsto X$ the identity map, $\fghat_F$ is isomorphic to the affine Kac-Moody algebra $\fghat$ associated to $(\fg, \langle, \rangle)$, and the corresponding vertex algebra is the affine vacuum module. When $F$ is a torus $(\CC^{*})^n$, $\fghat_F$ is the $(n+1)$--toroidal algebra, and we thus obtain a natural vertex realization of toroidal algebras. 

In the remainder of this introduction, we highlight the main ingredients in this construction with more detail. 

\subsection{(Pre)factorization algebras}\label{sec: prefact1}

The formalism of (pre)factorization algebras was developed by Kevin Costello and Owen Gwilliam in \cite{CG} to describe the algebraic structure of observables in quantum field theory as well as their symmetries. Roughly speaking, a prefactorization algebra $\cF$ on a manifold $X$ assigns to each open subset $U \subset X$ a cochain complex $\cF(U)$, and to each inclusion 
\[
U_1 \sqcup U_2 \sqcup \cdots \sqcup U_n \subset V
\]
of disjoint open subsets $U_i$ of $V$, a map 
\begin{equation} \label{str_map}
m^{U_1, \cdots, U_n}_V: \cF(U_1) \otimes \cdots \cdots \cF(U_n) \mapsto \cF(V)
\end{equation}
subject to some natural compatibility conditions. If $\cF$ is the prefactorization algebra of observables in a quantum field theory, the cohomology groups $H^i (\cF(U))$ can be interpreted as the observables of the theory on $U$ as well as their (higher) symmetries. This structure is reminiscent of a multiplicative cosheaf, and just as in the theory of sheaves/cosheaves a gluing axiom distinguishes factorization algebras from mere prefactorization algebras. 

An important source of prefactorization algebras is the \emph{factorization envelope} construction, which proceeds starting with a fine sheaf $\mc{L}$ of dg (or $L_{\infty}$ algebras) on $X$.  Denoting by $\mc{L}_c$ the cosheaf of compactly supported sections of $\mc{L}$, we have maps
\begin{equation} \label{ext1}
 \oplus^n_{i=1} \mc{L}_c (U_i) \cong \mc{L}_c(U_1 \cup \cdots \cup U_n) \mapsto \mc{L}_c (V)
\end{equation}
for disjoint opens $U_i \subset V$, where the map on the right is extension by $0$. Applying the functor $\Ch$ of Chevalley chains to (\ref{ext1}) yields maps
\[
\otimes^n_{i=1} \Ch(\mc{L}_c (U_i)) \mapsto \Ch(\mc{L}_c (V))
\]
The argument just sketched shows that the assignment 
\begin{equation} \label{envelope}
U \mapsto \Ch(\mc{L}_c (U))
\end{equation}
defines a prefactorization algebra. It is called the \emph{factorization envelope of $\mc{L}$} and denoted $\U \mc{L}$. 

Costello-Gwilliam showed that there is a close relationship between a certain class of prefactorization algebras on $X=\CC$ and vertex algebras. More precisely, if $\cF$ satisfies a type of holomorphic translation-invariance and is equivariant with respect to the natural $S^1$ action on $\CC$ by rotations, then the vector space
\begin{equation} \label{vertex_from_fact}
V({\cF}) := \bigoplus_{l \in \mathbb{Z}} H^{*}(\cF^{(l)}(\CC))
\end{equation}
has the structure of a vertex algebra, where $\cF^{(l)}(\CC)$ denotes the $l$-th eigenspace of $S^1$ in $\cF(\CC)$. 

\subsection{Affine Kac-Moody algebras, toroidal algebras, and associated vertex algebras}

Affine Kac-Moody Lie algebras \cite{KacInf} are a class of infinite-dimensional Lie algebras which play a central role in representation theory and conformal field theory. Given a finite-dimensional complex simple Lie algebra $\fg$, the corresponding affine Lie algebra $\fghat$ is the universal central extension of the loop algebra $\fg[z,z^{-1}] = \fg \tensor \mathbb{C}[z,z^{-1}]$ by a one-dimensional center $\mathbb{C}\kk$. I.e., there is a short exact sequence
\begin{equation} \label{KM_ext}
0 \to \mathbb{C}\kk \to \fghat \to \fg[z,z^{-1}] \to 0
\end{equation}
As a complex vector space $\fghat = \fg[z,z^{-1}] \oplus \mathbb{C}\kk$, with bracket
\[
[J \otimes f(z), J' \otimes g(z)] = [J,J']\otimes f(z) g(z) + \kk \langle J, J' \rangle Res_{z=0} f dg, 
\]
where $J, J' \in \fg$, and $\langle, \rangle$ denotes the Killing form. In physics, affine algebras appear as symmetries of conformal field theories (CFT's) such as the WZW model, and tools adapted from CFT, notably vertex algebras, play a crucial role in studying their representation theory. In particular, for each $K \in \mathbb{C}$, the \emph{vacuum module}
\[
V_{K}(\fghat) = \on{Ind}^{\fghat}_{\fghat_+} \mathbb{C}_K
\]
(where $\fghat_+ = \fg[z] \oplus \mathbb{C} \kk \subset \fghat$, and $\mathbb{C}_K$ denotes its one-dimensional representation on which the first summand acts trivially and $\kk$ acts by $K$ ) has the structure of a vertex algebra. The representation theory of $\fghat$ and $V_K (\fghat)$ are inextricably linked - $V_{K}(\fghat)$ picks out  interesting categories of representations of $\fghat$, and provides computational tools for studying these. 


One may consider a generalization of the universal central extension (\ref{KM_ext}) where $\mathbb{C}[z,z^{-1}]$ is replaced by an arbitrary commutative $\C$--algebra $R$. 
That is, one begins instead with the Lie algebra $\fg_R = \fg \tensor R$ with Lie bracket defined by 
\[
[J \tensor r, J' \tensor s] = [J,J'] \tensor rs .
\]
It is shown in \cite{Kassel} that $\fg_R$ has a central extension of the form 
\[
0 \to \Omega^1_R / \d R  \to \Hat{\fg}_R \to \fg_R \to 0.
\]
where $\Omega^1_R$ denotes the module of Kahler differentials, and $d: R \to \Omega^1_R $ is the universal derivation. Thus
$$\fghat_R \cong \fg \otimes R \oplus  \Omega^1_R / \d R,  $$
as a vector space, with bracket
$$ [J \tensor r, J' \tensor s] = [X,Y] \tensor rs + \overline{\langle X, Y \rangle r ds}, $$ where $\overline{\omega}$ denotes the class of $\omega \in \Omega^1_R$ in $ \Omega^1_R / \d R$. 
When $\fg$ is simple, the central extension $\fghat_R$ is universal. 

An important example occurs when $R = \C[t_1^{\pm 1}, \cdots, t_n^{\pm 1}]$ - the ring of algebraic functions on the $n$--torus. In this case $\fghat_R$ is a higher loop generalization of $\fghat$ called the  $n$--toroidal algebra. We note that for $n > 1$, the central term $ \Omega^1_R / \d R$ is infinite-dimensional over $\CC$.

It is a natural question whether one can associate to $\fghat_R$, and in particular to toroidal algebras, a natural vertex algebra $V(\fghat_R)$ along the same lines that $V_k (\fg)$ is associated to $\fghat$. In this paper, we provide an affirmative answer in the case when $R=A[z,z^{-1}]$ for some $\CC$--algebra $A$, and  give a geometric construction of $V(\fghat_R)$ in terms of prefactorization algebras.

\subsection{Prefactorization algebras from holomorphic fibrations}


In this paper we construct prefactorization algebras starting with two pieces of data: 

\begin{itemize}
\item A locally trivial holomorphic fibration $\pi: E \rightarrow X$ of complex manifolds with fiber $F$.
\item A Lie algebra $(\fg, \langle, \rangle)$ with invariant bilinear form.
\end{itemize}

We begin with a sheaf of dg Lie algebras (DGLA's) on $X$
\[
\fg_{\pi} = (\fg \otimes \pi_* \Omega^{0,*}_E, \delbar)
\]
with bracket $$ [J \otimes \alpha, J' \otimes \beta] = [J,J'] \otimes \alpha \wedge \beta, \; \;  \; J, J' \in \fg, \alpha, \beta \in \pi_* \Omega^{0,*}_E.$$ 
$\fg_{\pi}$ has an $L_{\infty}$ central extension $\fghat_{\pi}$ whose underlying complex of sheaves is of the form $$ \fghat_{\pi} = \fg_{\pi} \oplus \mc{K}_{\pi},$$ with $\mc{K}_{\pi}$ a certain three-term complex.  Our prefactorization algebra is $$\cF_{\pi, \fg} := \Ch(\fghat_{\pi, c}),$$ where $\fghat_{\pi, c}$ denotes the cosheaf of sections with compact support. This is   an instance of the factorization envelope construction (\ref{envelope}) described above. 

When $F$ is a smooth affine complex variety, and $E = X \times F$ is a trivial fibration, we may pass to a somewhat "smaller" prefactorization envelope $\sG^{alg}_{\pi, \fg}$
built starting with the sheaf of DGLA's
\[
(\fg \otimes H^0(F, \mc{O}) \otimes \Omega^{0,*}_X, \delbar)
\]
There is a map 
\begin{equation} \label{GtoF}
\sG^{alg}_{\pi, \fg} \to   \cF_{\pi, \fg}
\end{equation} induced by the map of DGLA's
\[
(\fg \otimes H^0(F, \mc{O}) \otimes \Omega^{0,*}_X, \delbar) \to (\fg \otimes \pi_* \Omega^{0,*}_E, \delbar)
\]
$\sG^{alg}_{\pi, \fg}$ is more manageable from technical standpoint, as it avoids certain analytic complications involving completions of Dolbeault cohomology of products. We note that the map (\ref{GtoF}) depends on a choice of trivialization of $E$. 

When $X=\CC$ (and $E$ is necessarily trivial), we may apply Costello-Gwilliam's result above to the prefactorization algebra $\sG^{alg}_{\pi, \fg}$,  recovering a vertex algebra $V({\sG^{alg}_{\fg, \pi}})$ as in \ref{vertex_from_fact}. It has the following simple description. Let $\fghat_F$ denote universal central extension of the Lie algebra $\fg \otimes H^0(F, \mc{O}_F)[z,z^{-1}]$ (i.e. this is $\fghat_R$ above, where $R=H^0(F, \mc{O}_F)[z,z^{-1}]$), and let $\fghat^+_{F}$ be the sub-algebra corresponding to the non-negative powers of $z$. Then as a representation of $\fghat_F$
\begin{equation} \label{VF}
V({\sG^{alg}_{\fg, \pi}}) \simeq \on{Ind}^{\fghat}_{\fghat^+_F} \CC,
\end{equation}
where $\CC$ is the trivial representation of $\fghat^+_F$. 

We may view these results as follows. When $X$ is a (arbitrary) Riemann surface, and $p \in X$ a point, we may choose a coordinate $z$ centered at p, and a local trivialization of $E$ near p. The cohomology prefactorization algebra $H^*(\cF_{\pi, \fg})$ is then locally modeled by the vertex algebra $V(\fghat_F)$ via the dense inclusion (\ref{GtoF}).  

\subsection{Outline of paper}

In section (\ref{Lie_section}) we recall universal central extensions, the construction of $\fghat_R$, and vertex algebras. We also show how to associate to the algebra $R=A[z,z^{-1}]$, where $A$ is a commutative $\CC$--algebra a vertex algebra generalizing the affine vacuum module. Our later geometric construction will be a special case of this. In section (\ref{fact_alg_section}) we recall some basic facts about prefactorization algebras. The construction of $\cF_{\fg,\pi}$  and the related prefactorization algebra $\sG^{alg}_{\fg, \pi}$ happens in section (\ref{main_construction_section}). Finally, in section (\ref{one_dim_section}) we consider the special case when $X=\CC$, and relate $\sG^{alg}_{\fg, \pi}$ to the vertex algebra $V(\fghat_F)$.

\medskip

\noindent {\bf Acknowledgements:} M.S. would like to thank Kevin Costello and Owen Gwilliam for patiently answering a number of questions and making several valuable suggestions. He also gratefully acknowledges the support of a Simons Collaboration Grant during the course of this project.

\section{Lie algebras and vertex algebras} \label{Lie_section}

\subsection{Conventions}

An $L_{\infty}$ structure on a graded vector space $\fh$ is the data of cohomological degree $1$ coderivation $\wt{d}$ of cofree cocommutative coalgebra
\[
\on{Sym}(\fh[1])
\]
satisfying $\wt{d}^2 = 0$. 
We may write $\wt{d} = \sum^{\infty}_{m=1} l_m $, where 
\[
l_m : \fh^{\otimes m} \to \fh[2-m]
\]
and the $l_m's$ are extended to the symmetric coalgebra as coderivations.  

Given such a square zero coderivation $\wt{d}$ the cochain complex
\[
\Ch (\fg) = \left(\on{Sym}(\fh[1]), \wt{d}\right)
\]
is called the {\em Chevalley-Eilenberg chain complex}. 
In the case of an ordinary Lie algebra, this complex computes Lie algebra homology. 

An ordinary Lie algebra corresponds to the case where $l_m=0$ for $m \neq 2$, and $l_2$ is the Lie bracket.
Differential graded Lie algebras correspond to the case where $l_m = 0$ for $m > 2$. 
We note that $(\fh, l_1)$ has the structure of a cochain complex, and $H^*(\fh, l_1)$ the structure of a graded Lie algebra. 
We refer to $\Ch(\fh, \wt{d})$ as the \emph{complex of Chevalley-Eilenberg chains} (even though it's cohomologically graded) since it computes Lie algebra homology. 
Its dual 
\[
\mc{C}^{*, Lie}(\fh)=\on{Sym}(\fh^*[-1], \wt{d^*})
\]
is the \emph{complex of Chevalley-Eilenberg cochains}. 
It computes Lie algebra cohomology.
For more on $L_{\infty}$ algebras, we refer the reader to \cite{KonSoi, LodVal}.

\subsection{Central extensions}

Let $\fg$ be a complex Lie algebra equipped with an invariant bilinear form $\ip$.
Also, fix a commutative $\CC$-algebra $R$.
Then, $ \fg_{R} := \fg \otimes_{\CC} R$ carries a natural complex Lie algebra structure with bracket
\[
[J \otimes r, J' \otimes s] = [J,J'] \otimes rs
\]
where $J,J' \in \fg$ and $r,s \in R$.
It is shown by Kassel ~\cite{Kassel} that there exists a universal central extension of the form
\[
0 \to H_2^{\rm Lie}(\fg_R) \to \Hat{\fg}_R \to \fg_R \to 0.
\]
Furthermore, when $\fg$ is simple and $\ip$ the Killing form, there is an isomorphism of the Lie algebra homology $H_2(\fg_R) \cong \Omega^1_{R/\CC} / \d R$ where $\Omega^1_{R/\CC}$ is the $R$-module of K\"{a}hler differentials of $R/\CC$ and $\d: R \to  \Omega^1_R$ is the universal derivation.
The bracket on 
\begin{equation} \label{uce}
\fghat_R \cong \fg \otimes R \oplus  \Omega^1_R / \d R, 
\end{equation}
 is given by 
\begin{align*} 
 [J \tensor r, J' \tensor s] &= [J,J'] \tensor rs + \overline{\langle J, J' \rangle r ds} \\
 				&= [J,J'] \tensor rs + \frac{1}{2} \overline{\langle J, J' \rangle (rds - s dr) }
\end{align*}				
where $\overline{\omega}$ denotes the class of $\omega \in \Omega^1_R$ in $ \Omega^1_R / \d R$. We will find the second form of the central cocycle more convenient to use. 

\begin{eg}

Let $n \geq 0$ be an integer.
An important class of examples is obtained by taking
\[
R := \CC[t^{\pm 1}_0, \cdots, t^{\pm -1}_n] .
\]
This is the algebra of functions on the $(n+1)$-dimensional algebraic torus. 

When $n=0$, the vector space $\Omega^1_R / \d R$ is one-dimensional with an explicit isomorphism given by the residue
\[
{\rm Res} : \Omega^1_R / \d R \xto{\cong} \CC . 
\]
The resulting Lie algebra $\fghat_R$ is the ordinary affine Kac-Moody algebra usually denoted by $\Hat{\fg}$. 
For $n \geq 1$, the vector space $\Omega^{1}_R / \d R$ is infinite dimensional. 
Indeed, let us denote $k_i = t_i^{-1} \d t_i$. 
The space $\Omega^1_{R} / \d R$ is generated over the ring $\CC[t_0^{\pm 1}, \ldots, t_n^{\pm 1}]$ by the symbols $k_0,\ldots, k_n$ subject to the relation
\[
\sum_{i = 0}^n m_i t_0^{m_0} \cdots t_n^{m_n} k_i = 0
\]
where $(m_0,\ldots, m_n)$ is any $n$-tuple of integers.
The Lie algebra $\Hat{\fg}_R$ is called the $(n+1)$-{\em toroidal} Lie algebra associated to $\fg$.

\end{eg}


It will be useful for us to have an $L_\infty$-model for the Lie algebra $\fghat_{R}$. 
This model amounts to replacing the vector space $\Omega^1_R / \d R$ appearing as the central term by the cochain complex
\[
\cK_{R} = \Ker (\d) [2] \to R[1] \xto{\d} \Omega^{1}_{R} .
\]
Just as the Lie algebra $\fghat_R$ is a central extension of $\fg_R = \fg \tensor R$, the $L_\infty$ model we wish to construct is a central extension of $\fg_R = \fg \tensor R$ by the cochain complex $\cK_R$. 
 
The central extension is determined by a cocycle $\phi \in C^*(\fg_R, \cK_R)$ of total degree two.
The cocycle is of the form $\phi = \phi^{(0)} + \phi^{(1)}$ where
\[
\begin{array}{ccccc}
\phi^{(1)} & : & (\fg_{R})^{\otimes 2} & \rightarrow & \Omega^1_{R} \\
& & (J \otimes r) \otimes (J' \otimes s) & \mapsto & \frac{1}{2} \langle J, J' \rangle (r \d s - s \d r)
\end{array}
\]
and 
\[
\begin{array}{ccccc}
\phi^{(0)} & : & (\fg_{R})^{\otimes 3} & \rightarrow & R \\
& & (J \otimes r)\otimes(J' \otimes s) \otimes (J'' \otimes t) & \mapsto & \frac{1}{2}  \langle [J,J'], J'' \rangle rst
\end{array}
\]

\begin{lemma}
The functional $\phi$ defines a cocycle in $C^*(\fg_{R}, \cK_{R}) $ of total degree two.
\end{lemma}
\begin{proof}
The differential in the cochain complex $C^*(\fg_R , \cK_R)$ is of the form $\d + \d_{CE}$ where $\d$ is the de Rham differential defining the complex $\cK_R$, and $\d_{CE}$ is the Chevalley-Eilenberg differential encoding the Lie bracket of $\fg_R$.
It is immediate that $\d \phi^{(1)} = 0$, $\d_{CE} \phi^{(0)}=0$ by the Jacobi identity for $\fg$ and invariance of $\ip$, and $\d \phi^{(0)} + \d_{CE} \phi^{(1)} = 0$ by direct calculation.
Thus $(\d_{CE} + \d) \phi = 0$ as desired.
\end{proof}

The cocycle $\phi$ defines an $L_{\infty}$ central extension
\[
\cK_R \to \fgtil_R \to \fg_R .
\]
As a vector space, $\fgtil_{R} = \fg_R \oplus \cK_R$, and the $L_\infty$ operations are defined by $\ell_1 = \d, \ell_2 = [\cdot,\cdot]_{\fg_R} + \phi^{(1)}$, and $\ell_3 = \phi^{(0)}$.
The following is immediate from our definitions:

\begin{lemma}
There is an isomorphism of Lie algebras
$H^{*}(\fgtil_{\R}, \ell_1) = \fghat_{\R}.$
\end{lemma} 
\begin{proof}
The cohomology of $\Tilde{\fg}_R$ is concentrated in degree zero, and isomorphic to $$\fg_R \oplus H^0(\cK_R) = \fg_R \oplus \Omega^1_R / \d R $$ as a vector space. 
Our definition of $\phi$ implies that 
\[
\phi^{(0)} ((J \otimes r) \otimes (J' \otimes s)) = \frac{1}{2} \langle J, J' \rangle (r ds - s dr) = rds \; mod \; dR  
\]
so that the resulting Lie bracket is the same as that of $\Hat{\fg}_R$. 
\end{proof}

\subsection{Vertex algebras}

We proceed to briefly recall the basics of vertex algebras and discuss an important class of examples, which will later be constructed geometrically via factorization algebras. We refer the reader to ~\cite{FBZ, Kac} for details. 

\begin{dfn}
A vertex algebra $(V,\vac, T, Y)$ is a complex vector
space $V$ along with the following data:
\begin{itemize}
\item A vacuum vector $\vac \in V$.
\item A linear map $T : V \to V$ (the translation operator).
\item A linear map $Y(-,z) : V \to {\rm End}(V)\llbracket z^{\pm 1}
  \rrbracket$ (the vertex operator). We write $Y(v,z) = \sum_{n \in \mathbb{Z}} A_n^v z^{-n}$
  where $A_n^v \in {\rm End}(V)$. 
\end{itemize} 
satisfying the following axioms:
\begin{itemize}
\item For all $v,v' \in V$ there exists an $N \gg 0$ such that $A_n^v
  v' = 0$ for all $n > N$. (This says that $Y(v,z)$ is a {\it field}
  for all $v$). 
\item (vacuum axiom) $Y(\vac, z) = {\rm id}_V$ and
    $Y(v,z)  |0\> \in v + z V \llbracket z \rrbracket$ for all
    $v \in V$. 
\item (translation) $[T,Y(v,z)] = \partial_z Y(v,z)$ for all $v \in
  V$. Moreover $T \vac = 0$. 
\item (locality) For all $v,v' \in V$, there exists $N \gg 0$ such
  that 
\[
(z-w)^N[Y(v,z),Y(v',w)] = 0
\]
in ${\rm End}(V) \llbracket z^{\pm 1},w^{\pm 1}\rrbracket$. 
\end{itemize}
\end{dfn}

In order to prove that a given $(V,\vac,T,Y)$ forms a vertex algebra, the following "reconstruction" or "extension" theorem is very useful. It shows that any collection of local fields generates a vertex algebra in a suitable sense. 

\begin{theorem}[~\cite{FBZ}, ~\cite{DSK}] \label{rec_thm} Let $V$ be a complex vector space equipped with: an
  element $\vac \in V$, a linear map $T : V \to V$, a 
    set of vectors $\{a^s\}_{s \in S} \subset V$ indexed by a set $S$, and fields $A^s(z) =
    \sum_{n \in \mathbb{Z}} A_n^sz^{-n-1}$ for each $s\in S$ such that:
\begin{itemize}
\item For all $s \in S$, $A^s(z) \vac \in a^s + z V\llbracket
    z\rrbracket$;
\item $T \vac = 0$ and $[T,A^s(z)] = \partial_z A^s(z)$;
\item $A^s(z)$ are mutually local;
\item and $V$ is spanned by $\{A_{j_1}^{s_1} \cdots A_{j_m}^{s_m}
  |0\>\}$ as the $j_i's$ range over negative integers. 
\end{itemize}
Then, the data $(V,\vac, T,Y)$ defines a unique vertex algebra satisfying 
\[
Y(a^s,z) = A^s(z) .
\]
\end{theorem}

\begin{rmk}
The version stated above appears in ~\cite{DSK}, and is slightly more general than the version stated in ~\cite{FBZ}. 
\end{rmk}

\subsection{The vertex algebras $V(\fghat)$ and $V(\fghat_{\R})$}

A number of vertex algebras are constructed from vacuum representations of affine Lie algebras and their generalizations. We proceed to review the vertex algebra structure on the affine Kac-Moody vacuum module $V(\fghat)$ and extend the construction to vacuum representations of $\fghat_{R}$, where $R=A[t,t^{-1}]$ for some $\CC$-algebra $A$.

\subsubsection{$V(\fghat)$}
Let $\fghat = \fg[t,t^{-1}] \oplus \CC \kk$ be the affine Kac-Moody algebra, $\fghat^{+} = \fg[t]$ denote the positive sub-algebra, and $\CC$ denote the trivial representation of $\fghat^{+}$.
For $J \in \fg$, denote $J \otimes t^n$ by $J_n$, and $1 \in \CC$ by $\vac$

It is well-known (see for instance ~\cite{FBZ}) that the induced vacuum representation
\[
V (\fghat) := \on{Ind}^{\fghat}_{\fghat^+} \CC := U(\fghat) \otimes_{U(\fg[t])} \CC
\] 
has a $\CC[\kk]$-linear vertex algebra structure, which is generated, in the sense of the above reconstruction theorem, by the fields
\[
J^i (z) := Y(J^{i}_{-1} \vac, z) = \sum_{n \in \mathbb{Z}} J^i_n z^{-n-1},
\]
where $\{ J^{i} \}$ is a basis for $\fg$. These satisfy the commutation relations
\[
[J^i(z), J^{k}(w)] = [J^{i}, J^{k}](w) \delta(z-w) +  \langle J^{i}, J^{k} \rangle \kk \partial_w \delta(z-w)
\]
where 
\[
\delta(z-w) = \sum_{m} z^{m} w^{-m-1} 
\]
The translation operator $T$ is determined by the properties
\[
T \vac =0, [T, J^i_n] = -n J^{i}_{n-1}.
\]

\begin{rmk}\label{rmk: level}
The construction above produces a generic version of the affine Kac-Moody vacuum module, in the sense that $\kk$ is not specialized to be a complex number. In the vertex algebra literature one typically specifies a level $K \in \mathbb{C}$, and defines $$V_K (\fghat) := \on{Ind}^{\fghat}_{\fg[t] \oplus \CC \kk} \CC,$$ where $\CC$ denotes the one-dimensional representation of $\fg[t] \oplus \CC \kk$ on which the first factor acts by $0$ and $\kk$ acts by $K$. We have an isomorphism 
\[
V(\fghat) / I \simeq V_K (\fg)
\]
where $I$ is the vertex ideal generated by $K \vac - \kk \vac$. $V(\fg)$ can therefore be viewed as a family of vertex algebras over $\on{spec}(\CC[\kk])$, with fiber $V_K(\fg)$ at $\kk=K$. 
\end{rmk}

\subsubsection{The generalized toroidal vertex algebra} \label{VR}

In this section we generalize the construction of the affine Kac-Moody vacuum module above to the Lie algebra $\fghat_R$, for $R=A[t,t^{-1}]$,  where $A$ is a commutative $\CC$--algebra. The construction specializes to $V(\fghat)$ for $A = \CC$. 


Let $A$ be a commutative $\CC$-algebra, $R=A[t,t^{-1}] := A \otimes \CC[t,t^{-1}]$, and $\fghat_{R}$ the Lie algebra (\ref{uce}). 
We have a Lie subalgebra 
\[
\fghat^{+}_{R} := \fg \tensor A [t] \oplus \Omega^1_{A[t]}/ d A[t] \hookrightarrow \fghat_R .
\]
Let
\begin{equation}
V(\fghat_{R}) := \on{Ind}^{\fghat_{R}}_{\fghat^+_{R}} \CC
\end{equation}
where $\CC$ denotes the trivial representation of $\fghat^+_R$. Our goal is to define the structure of a vertex algebra on $V(\fghat_R)$. 

The vacuum vector is simply $\vac := 1 \in \CC.$
The fields of the vertex algebra split into three classes and are defined as follows. 

\begin{align}
J_u(z) := Y(J \otimes u t^{-1} \vac,z ) & := \sum_{n \in \mathbb{Z}} (J  \otimes u t^{n}) z^{-n-1}, \\
K_{u \frac{dt}{t}} :=  Y(t^{-1} u dt \vac,z ) & := \sum_{n \in \mathbb{Z}} ( u t^{n-1} dt )  z^{-n}, \\
K_{t^{-1} \omega} := Y(t^{-1}  \omega \vac,z ) & := \sum_{n \in \mathbb{Z}}( t^{n} \omega )    z^{-n-1} 
\end{align}
where $J\in \fg, u \in A, \omega \in \Omega^{1}_A$. 

The commutation relations between these fields are easily checked to be
$$
[J^1_u (z), J^2_v (w)] = \left( [J^1, J^2]_{uv} (w) + \langle J^1, J^2 \rangle  K_{t^{-1} u dv} (w)      \right) \delta(z-w) + \langle J^1, J^2 \rangle K_{uv \frac{dt}{t}} (w) \partial_w \delta(z-w)
$$
with all other commutators $0$. 

The operator $T$, corresponding to the Lie derivative $L_{-\partial_t}$, is defined by
\[
T \vac =0, \; [T, J^i \otimes f t^n] = -n J^i \otimes f t^{n-1}, \; [T, f t^n dt] = -n f t^{n-1} dt, \; [T, t^n \omega] = -n t^{n-1} \omega
\]

\begin{theorem}
The above field assignments, together with $T$ equip $V(\fghat_{\R})$ with the structure of a vertex algebra.
\end{theorem}

\begin{proof}
We begin by checking that the field assignment above is well-defined. This amounts to verifying that $Y(d(t^{-1}u)\vac,z ) =0$.  We have 
\begin{align*}
Y(d(t^{-1}u) \vac, z) &= Y(t^{-1} du \vac, z) - Y(f t^{-2} u \vac, z) \\
&= Y(t{-1} du \vac, z) - Y([T, t^{-1}u dt  ] \vac, z) \\
&= Y(t^{-1} du \vac, z) - \partial_z Y( t^{-1} u dt \vac, z) \\
&= \sum_{n} (t^n du + n t^{n-1} u) z^{-n-1} = \sum_{n} d(t^n u) z^{-n-1} = 0
\end{align*}
The result now follows by applying the reconstruction theorem \ref{rec_thm} to $V(\fghat_\R)$ and the fields $\{ J_u (z), K_{u \frac{dt}{t}}, K_{t^{-1} \omega} \}$ for $J \in \fg, f \in A, \omega \in \Omega^{1}_A $.
\end{proof}

The map sending a $\CC$-algebra $A$ to $V(\fghat_{R})$, with $R=A[z,z^{-1}]$ has a number of pleasing properties.  Denote by $\CC-\on{Alg}$ the category of commutative $\CC$-algebras and $\on{Vert}$ the category of vertex algebras.  We have the following result:

\begin{prop} \label{functoriality}
The map
\begin{align*}
\on{\CC-Alg} & \to \on{Vert} \\
A & \mapsto V(\fghat_{A[z,z^{-1}]})
\end{align*}
defines a functor. 
\end{prop}

\begin{proof}
We must check that a homomorphism of $\CC$-algebras $\psi: A \to B$ induces a vertex algebra homomorphism $ \wt{\psi}: V(\fghat_{A[t,t^{-1}]}) \to V(\fghat_{B[t,t^{-1}]})$. 
We begin by constructing a Lie algebra homomorphism $$\ol{\psi}: \fghat_{A[t,t^{-1}]} \to \fghat_{B[t,t^{-1}]}.$$
Recall that a $\CC$-algebra homomorphism $\sigma: R \to S$ induces a map on Kahler differentials (as vector spaces) $\sigma_*: \Omega_R \to \Omega_S$ given by 
$ \sigma_* ( r dr') = \sigma(r) d \sigma(r') $, which sends exact elements to exact elements, inducing a map
\[
\sigma_* : \Omega_R / dR \to \Omega_S / dS
\]
Extending $\psi: A \to B$ to a homomorphism (abusively also denoted $\psi$) $\psi: A[t,t^{-1}] \to B[t,t^{-1}]$, and taking $\sigma=\psi$ yields a map
\[
\psi_*: \Omega_{A[t,t^{-1}]} / d A[t,t^{-1}] \to  \Omega_{B[t,t^{-1}]} / d B[t,t^{-1}]
\]
Now, $\ol{\psi}: \fghat_{A[t,t^{-1}]} \to \fghat_{B[t,t^{-1}]}$ is defined by 
\[
\ol{\psi}( J u t^n + \omega) = J \psi(u) t^n + \psi_* \omega \; \; \;  \; \textrm{ where } J \in \fg, \;  u \in A, \; \omega \in \Omega_{A[t,t^{-1}]} / d A[t,t^{-1}] 
\]
and easily checked to be a Lie algebra homomorphism. Finally, $$ \wt{\psi}: V(\fghat_{A[t,t^{-1}]}) \to V(\fghat_{B[t,t^{-1}]})$$ may be defined on the generating fields $J_u (z), K_{u \frac{dt}{t}}, K_{t^{-1} \omega}$ in the obvious way by:
\begin{align*}
\wt{\psi} (J_u (z)) :=  \sum_{n \in \mathbb{Z}} (J  \otimes \psi(u) t^{n}) z^{-n-1}, \\
\wt{\psi} (K_{u \frac{dt}{t}} (z)) := K_{\psi_* (u \frac{dt}{t}) }(z) =  \sum_{n \in \mathbb{Z}} ( \psi(u) t^{n-1} dt )  z^{-n}, \\
\wt{\psi} (K_{t^{-1} \omega}(z)) := K_{ \psi_* (t^{-1} \omega)}(z) =  \sum_{n \in \mathbb{Z}}( t^{n} \psi_* \omega )    z^{-n-1} 
 \end{align*}
where $u \in A, \omega \in \Omega_A$. One easily checks that $$ \wt{\psi}([J_u(z), J'_v (w)]) = [\wt{\psi}(J_u (z)), \wt{\psi}(J'_v (w)) ], $$ which shows that $\wt{\psi}$ respects the only non-trivial OPE among the generating fields. It follows that $\wt{\psi}$ is a vertex algebra homomorphism. 
\end{proof}

If $A$ is any $\CC$-algebra, we may apply this result to the structure map $\CC \to A$, to obtain:

\begin{cor}
The structure map $\CC \to A$ induces an embedding of vertex algebras $$V(\fghat ) \to V(\fghat_{A[t,t^{-1}]}).$$
\end{cor}

\begin{rmk}
As explained in Remark \ref{rmk: level}, we have $$V_K (\fg) \simeq V(\fghat_{\CC[t,t^{-1}]})/I, $$ where $V_K (\fg)$ denotes the "usual" affine vacuum module at level $K \in \CC$, and $I$ is the vertex ideal generated by ${\bf k} \vac - K \vac$. When $K \neq - h^{\vee}$ (where $h^{\vee}$ denotes the dual Coxeter number of $\fg$), $V_K (\fg)$ is a conformal vertex algebra with conformal Segal-Sugawara vector 
\[
S = \frac{1}{2(K + h^{\vee})} \sum^d_{i=1} (J_{i} \otimes t^{-1})(J_i \otimes t^{-1}) \vac, 
\]
where $\{ J_i \}^d_{i=1}$ an orthonormal basis of $\fg$ with respect to the invariant pairing $\langle \bullet, \bullet\rangle$. It follows that when $K \neq - h^{\vee}$, $S$ defines a conformal vector in $V(\fghat_{A[t,t^{-1}]}) / I'$, where $I'$ is the vertex ideal generated by $\frac{dt}{t} \vac - K \vac$. 
 
\end{rmk}

As another consequence of Prop. \ref{functoriality}, we note that any algebra automorphism $\psi: A \to A$ induces a vertex algebra automorphism of $V(\fghat_{A[t,t^{-1}]})$:

\begin{cor}
There is a natural group homomorphism
\begin{align*}
\on{Aut}(A) & \to \on{Aut}(V(\fghat_{A[t,t^{-1}]})) \\
\psi & \to \wt{\psi} .
\end{align*}
\end{cor}

\section{(Pre)factorization algebras and examples} \label{fact_alg_section}

In this section we recall basic notions pertaining to pre-factorization algebras. We refer the reader to \cite{CG} for details. 

Let $X$ be a smooth manifold, and $\symcat$ a symmetric monoidal category. 

\begin{dfn}
A {\it prefactorization algebra} $\Fcal$ on $X$ with values in $\symcat$ consists of the following data:
\begin{itemize}
\item for each open $U \subset M$, an object $\Fcal(U) \in \symcat $,
\item for each finite collection of pairwise disjoint opens $U_1,\ldots,U_n$ and an open $V$ containing every $U_i$, a morphism
\begin{equation} \label{struct_maps}
m^{U_1, \cdots, U_n}_V: \Fcal(U_1) \otimes \cdots \otimes \Fcal(U_n) \to \Fcal(V),
\end{equation}
\end{itemize}
and satisfying the following conditions:
\begin{itemize}
\item composition is associative, so that the triangle
\[
\begin{tikzcd}
\bigotimes_i \bigotimes_j \Fcal(T_{ij}) \arrow{rr} \arrow{rd} &&\bigotimes_i \Fcal(U_{i}) \arrow{ld} \\
&\Fcal(V)&
\end{tikzcd}
\]
commutes for any disjoint collection $\{U_i\}$ contained in $V$, and disjoint collections $\{T_{ij}\}_j \subset U_i$ 
\item the morphisms $m^{U_1, \cdots, U_n}_V$ are equivariant under permutation of labels, so that the triangle
\[
\begin{tikzcd}
\Fcal(U_{1}) \otimes \cdots \otimes \Fcal(U_n) \arrow{rr}{\simeq} \arrow{rd} && \Fcal(U_{\sigma(1)}) \otimes \cdots \otimes \Fcal(U_{\sigma(n)}) \arrow{ld}\\
&\Fcal(V)&
\end{tikzcd}
\]
commutes for any $\sigma \in S_n$.
\end{itemize}
\end{dfn}

In this paper, we will take the target category $\symcat$ to be $\Vect$, $\dgVect$, or their smooth enhancements $\DVS$ described below.  

A {\it factorization algebra} is
a prefactorization algebra satisfying a descent (or gluing) axiom with respect to a class of special covers called \emph{Weiss covers}. In this paper, we will not be concerned with verifying that our constructions satisfy this additional property, and refer the interested reader to ~\cite{CG} for details. 

\begin{eg}[\cite{CG}] \label{ass_eg}
Given an associative algebra over $\CC$, one can construct a prefactorization algebra $\cF_{A}$ in $\Vect$ on $\RR$ by declaring $\cF_{A}(I) = A$ for a connected open interval $I \subset \RR$, and defining the structure maps in terms of the multiplication on $A$. For instance, if $I=(a,b), J=(c,d), K=(e,f)$, with $e < a < b < c < d < f$, the structure map is
\begin{align*}
\cF_A (I) \otimes \cF_A (J) & \mapsto \cF_A (K) \\
a \otimes b &\mapsto ab
\end{align*}
$\cF_A$ has the property that it is \emph{locally constant}, in the sense that if $I \subset I'$ are connected intervals, then $\cF_A (I) \mapsto \cF_A (I')$ is an isomorphism. It is shown in \cite{CG} that locally constant prefactorization algebras on $\RR$ in $\Vect$ correspond precisely to associative algebras. 
\end{eg} 

Prefactorization algebras can be pushed forward under smooth maps as follows. Suppose $f: X \mapsto Y$ is a smooth map of smooth manifolds, and $\cF$ a prefactorization algebra on $X$. One then defines the prefactorization algebra $f_* \cF$ on $Y$ by
\[
f_* \cF (U) := \cF(f^{-1}(U))
\] 
The structure maps of $f_* \cF$ are defined in the obvious way. 

If $\cF, \cG$ are prefactorization algebras on $X$ with values in $\symcat$, then a morphism $\phi: \cF \to \cG$ is the data of maps $$\phi_U : \cF(U) \to \cG(U)  \in \on{Hom}_{\symcat} (\cF(U), \cG(U))$$ for each open $U \subset X$, compatible with all structure maps (\ref{struct_maps}).

\subsection{(Pre)factorization envelopes}

Our construction of factorization algebras from holomorphic fibrations is an instance of the \emph{factorization envelope} construction, which we proceed to briefly review following \cite{CG}. 

Let $\mc{L}$ be a fine sheaf of $L_{\infty}$ algebras, and $\mc{L}_c$ its associated cosheaf of sections with compact support. The \emph{factorization envelope of $\mc{L}$}, $\U \mc{L}$ is the complex of Chevalley chains of $\mc{L}_c$.  In other words, for each open $U \subset X $
\begin{equation} \label{fact_envelope}
\U \mc{L} (U) :=  \Ch(\mc{L}_c (U))
\end{equation}
The factorization structure maps are given explicitly as follows. Let $U_1, \cdots, U_k$ be disjoint open subsets of an open $V \subset X$. The cosheaf $\mc{L}_c$  induces a map of $L_{\infty}$--algebras
\[
\oplus^k_{i=1} \mc{L}_c (U_i) \mapsto \mc{L}_c (V)
\]
Applying the Chevalley chains functor (which sends sums to tensor products) to this sequence yields maps
\[
\otimes^k_{i=1} \Ch ( \mc{L}_c (U_i)) \mapsto \Ch (\mc{L}_c (V)).
\]
The following is proved in \cite{CG}

\begin{theorem}
If $\mc{L}$ is a fine cosheaf of $L_{\infty}$ algebras, then $\U \mc{L}$ is a prefactorization algebra in $\dgVect$. The cohomology $H^*(\U \mc{L})$ is a prefactorization algebra in $\Vect$. 
\end{theorem}

\begin{eg}[\cite{CG}] \label{gdeRham}
Let $X=\mathbb{R}$, and $\fg_{dR} := (\fg \otimes \Omega^{*}_{\mathbb{R}}, d_{dR})$ the sheaf of DGLA's on $\mathbb{R}$ obtained by tensoring $\fg$ with the de Rham complex.
The factorization envelope $\U(\fg_{dR})$ is locally constant, and the cohomology factorization algebra $H^*( \U (\fg \otimes \Omega^*))$ is a locally constant factorization algebra in $\Vect$, corresponding to the enveloping algebra $\mc{U}(\fg)$ as in Example (\ref{ass_eg}).  
\end{eg}

\begin{eg}[\cite{GWkm}] \label{gDolbeault}
Let $X = \CC^n$, and $(\fg \otimes \Omega^{0,*}_{\CC^n}, \delbar )$ the sheaf of DGLA's on $\CC^n$ obtained by tensoring $\fg$ with the Dolbeault complex of forms of type $(0,q), q \geq 0$. As explained below, when $n=1$, the factorization algebra $\U(\fg \otimes \Omega^{0,*}_{\CC}, \delbar )$ allows one to recover the affine vertex algebra $V(\fghat)$ (at level 0). 
\end{eg}

\subsection{Differentiable vector spaces} \label{DVS}

The prefactorization algebras considered in  this paper typically assign to each open subset $U \subset X$  a cochain complex of infinite-dimensional vector spaces. This is apparent already in the example \ref{gdeRham} above, where the graded components of $\U(\fg_{dR}) = \U (\fg \otimes \Omega^*)(U)$  for $U \subset \mathbb{R}$ are tensors in $\fg \otimes \Omega^*(U)_c$. The structure maps (\ref{struct_maps}) are thus multilinear maps between such complexes. In order to formulate the notion of translation-invariance for prefactorization algebras in the next section, we will have to discuss what it means for these to depend smoothly on the positions of the open sets $U_i \subset X$. This raises some functional-analytic issues, which in turn complicate homological algebra involving these objects.
 
 In \cite{CG} these technical issues are resolved by introducing the category $\on{DVS}$ of \emph{Differentiable Vector Spaces} together with certain sub-categories. $\on{DVS}$ provides a flexible framework within which one can discuss smooth families of smooth maps between infinite-dimensional cochain complexes parametrized by auxilliary manifolds, and carry out homological constructions. We briefly sketch this category below, and refer to \cite{CG} for all details. 
 
\begin{dfn}
Let $C^{\infty}$ denote the sheaf of rings on the site of smooth manifolds sending each manifold $M$ to the ring of smooth functions $C^{\infty}(M)$, and assigning to each smooth map $f: M \rightarrow N$ the pullback $f^{*}: C^{\infty}(N) \rightarrow C^{\infty}(M)$.  A \emph{$C^{\infty}$-module} $\cF$ is a sheaf of modules over $C^{\infty}$. In other words, $\cF$ assigns to each $M$ a $C^{\infty}(M)$-module $\cF(M)$, and to $f: M \rightarrow N$ a pullback map $\cF(f): \cF(N) \rightarrow \cF(M)$ of $C^{\infty}(N)$-modules. A \emph{differentiable vector space} is a $C^{\infty}$-module equipped with a flat connection. Explicitly, this means a flat connection 
\[
\nabla: \cF(M) \rightarrow \cF(M) \otimes_{C^{\infty}(M)} \Omega^1(M)
\]
for each manifold $M$, compatible with pullbacks. The objects of the category $DVS$ are differentiable vector spaces, and the morphisms $Hom_{DVS}(\sF, \sG)$ maps of $C^{\infty}$-modules intertwining the connections. 
\end{dfn}

Any locally convex topological vector space $V$ gives rise to a differentiable vector space as follows. There is a good notion of a smooth map from any manifold $M$ to $V$ introduced by Kriegl and Michor (see \cite{KM}), and we denote by $C^{\infty}(M,V)$ the space of such. The space $C^{\infty}(M,V)$ is naturally a $C^{\infty}(M)$-module, and carries a natural flat connection whose horizontal sections are constant maps $M \rightarrow V$. The assignment $M \rightarrow C^{\infty}(M,V)$  thus produces an object of $\DVS$.  
Multi-linear maps 
\begin{equation} \label{multi_linear}
\cF_1 \times \cF_2 \times \cdots \times \cF_r \mapsto \cG \; \; \cF_i, \cG \in \DVS
\end{equation}
equip $\DVS$ with the structure of a multi-category (or equivalently, a colored operad) by inserting the output of a multilinear map into another. We denote the space of such maps by $\DVS(\cF_1, \cdots, \cF_r \vert \cG)$. 

The multicategory $\DVS$ allows us to formulate the notion of a smooth family of multilinear operations parametrized by an auxiliary manifold $M$. For $\cF \in \DVS$, one first defines the mapping space $\bf{C}^{\infty}(M,\cF) \in \DVS$ as the differentiable vector space given by the assignment $N \mapsto \cF(N \times M)$.  As explained in \cite{CG}, it has a natural flat connection along $N$.

\begin{dfn} \label{smooth_families}
Let $\cF_1, \cdots, \cF_r, \cG \in \DVS$. A smooth family of multilinear operations $\cF_1 \times \cdots \cF_r \mapsto \cG$ parametrized by a manifold $M$ is by definition an element of 
\[
\DVS(\cF_1, \cdots, \cF_r \vert \bf{C}^{\infty}(M,\cG))
\] 
where $\bf{C}^{\infty}(M,\cG)$ is as explained in the preceding paragraph. 
\end{dfn}

$\DVS$ has several good properties. Among these are:

\begin{itemize}
\item $\DVS$ is complete and co-complete. 
\item $\DVS$ is a Grothendieck Abelian Category. 
\end{itemize}

The second property ensures that all standard constructions in homological algebra behave well in $\DVS$. This is in contrast to the category of topological vector spaces, which is not even Abelian. As the authors explain in \cite{CG}, this is because $\DVS$ has  essentially been defined as the category of sheaves on a site. 

Finally, we review some examples of differentiable vector spaces which will be useful to us. 

\begin{eg}
The following is an important example from \cite{CG}.
Suppose $p: W \rightarrow X$ is a vector bundle over the manifold $X$. Then $V=\Gamma(X, W)$ is naturally a Frechet space, and so locally convex. $C^{\infty} (M,V)$ is then identified with $\Gamma(M \times X, \pi^*_X E)$, where $\pi_X : M \times X \rightarrow X$ is the projection on $X$. In particular, taking $W$ to be the trivial bundle, we have $C^{\infty}(M, C^{\infty}(X))= C^{\infty}(M \times X)$.  The same line of reasoning shows that  the space of compactly supported sections of $W$, $V'=\Gamma_c (X, W)$ has a $\DVS$ structure. 
\end{eg}

\begin{eg} \label{direct_image_DVS}
The following generalization of the previous example will be useful in Sections \ref{main_construction_section}, \ref{one_dim_section}. Let $\pi: E \rightarrow X$ be a smooth map, and $p: W \rightarrow E$ a vector bundle on $E$. Denote by $\mc{W}$ the sheaf of smooth sections of $W$ on $E$. Then $V=\Gamma(X, \pi_* \mc{W})$ yields a differentiable vector space, with $C^{\infty}(M, V) = \Gamma(M \times X, \tilde{\pi}_{*} \wt{\mc{W}})$, where $\tilde{\pi} : E \times M \rightarrow X \times M$ is defined by $\tilde{\pi}(e,m) = (\pi(e), m)$, and $\wt{\mc{W}}$ denotes the sheaf of sections of $\pi^*_E W$ on $E \times M$, with $\pi_E : E \times M \rightarrow E$ the projection on the first factor. The connection is determined by the condition that the horizontal sections are those constant in the $M$ direction. When $E=X$ and $\pi = id_X $, this example reduces to the previous one. We may similarly equip $V' = \Gamma_c (X, \pi_* \mc{W})$  with a $\DVS$ structure. 
\end{eg}

\begin{eg} \label{tensor_DVS}
Suppose that $\cF \in \DVS$, and $V$ is any vector space (note that we don't specify a topology on $V$). Then the assignment $M \mapsto \cF(M) \otimes_{\mathbb{R}} V$, with the connection acting trivially on the $V$ factor,  yields an object of $\DVS$ which we denote $\cF_V$, When $V$ is finite-dimensional, this amounts to a finite direct sum of $\cF$. 
\end{eg}

 \subsubsection{Monoidal structures on $\DVS$}
 
 To discuss prefactorization algebras with values in $\DVS$, we must specify a symmetric monoidal structure, which is used in defining the structure maps (\ref{struct_maps}). Certain subtleties arise on this point, typical of the issues one encounters when trying to define tensor products of inifinite-dimensional topological vector spaces. We restrict ourselves to a few brief remarks, and refer the interested reader to \cite{CG} for details. 
 
 \begin{itemize}
 \item Given $\cF, \cG in \DVS$, we can define $\cF \otimes \cG$ as the sheafification of the presheaf $X \mapsto \cF(X) \otimes_{C^{\infty}(X)} \cG(X)$, equipped with the flat connection $\nabla^{\cF} \otimes Id + Id \otimes \nabla^{\cG}$. When $\cF=C^{\infty}(M), \cG=C^{\infty}(N)$, and $X=pt$ is a point, this yields $\cF \otimes \cG (pt) = C^{\infty}(M) \otimes_{\mathbb{R}} C^{\infty}(N)$. We call this symmetric monoidal structure the \emph{naive tensor product} in $\DVS$.
 \item The naive tensor product has certain shortcomings. Most importantly, it does not represent the space of multilinear maps (\ref{multi_linear}). In order to remedy this situation, a certain completed tensor product $\hat{\otimes}_{\beta}$ has to be introduced. This operation is only defined on a certain sub-category of $\DVS$ however. In the last example, we would obtain $\cF \hat{\otimes}_{\beta} \cG (pt) = C^{\infty}(M \times N)$. We will refer to this operation as the \emph{completed tensor product}. 
 \end{itemize}  

Using $ \hat{\otimes}_{\beta} $ rather than $\otimes$ is important if one wishes to obtain a factorization, rather than merely a prefactorization algebra. As we work with prefactorization algebras in this paper, the naive tensor product is adequate, and will be the symmetric monoidal structure on $\DVS$ throughout.

\subsection{Translation-invariant (pre)factorization algebras}

Our construction in Section \ref{main_construction_section}, when applied to the trivial fibration $E=F \times \CC^n \mapsto \CC^n$, produces a prefactorization algebra which is \emph{holomorphically translation-invariant}. This property will be used when extracting a vertex algebra in Section ~\ref{one_dim_section} in the case $n=1$. We proceed to briefly review this notion and refer the interested reader to \cite{CG} for details. 

\subsubsection{Discrete translation-invariance}
Suppose now that $\sF$ is prefactorization algebra on $\CC^n$ in the category of
complex vector spaces. $\CC^n$ acts on itself by translations. For an open subset $U \subset \CC^n$ and $x \in \CC^n$, let 
\[
\tau_x U := \{ y \in \CC^n \vert y-x \in U \}
\]
Clearly, $\tau_x (\tau_y U) = \tau_{x+y} U$. We say that $\sF$ is \emph{discretely translation-invariant} if we are given isomorphisms
\begin{equation} \label{phi}
\phi_x : \sF(U) \rightarrow \sF(\tau_x U)
\end{equation}
for each $x \in \CC^n$ compatible with composition and the structure maps of $\sF$. We refer to section 4.8 of  \cite{CG} for details. 

\begin{eg}
For any Lie algebra $\fg$,  $\U(\fg \otimes \Omega^{0,*}_{\CC^n})$ in Example ~\ref{gDolbeault} is discretely translation-invariant. 
\end{eg}

\subsubsection{Smooth and holomorphic translation-invariance}

A refined version of translation-invariance expresses the fact that the maps $\phi_x$, and hence the structure maps $m^{U_1, \cdots, U_n}_V$ depend smoothly/holomorphically on the positions of the open sets $U_i$. This notion is operadic in flavor. 

For $z \in \CC^n$ and $r > 0$ let $\PD(z,r)$ denote the polydisk
\[
\PD_r(z) = \{ w \in \CC^n \vert \vert w_i - z_i \vert < r,  1 \leq i \leq n \}
\]
and let 
\[
\PD(r_1, \cdots, r_k \vert s ) \subset (\CC^n)^k 
\]
denote the open subset $(z_1, \cdots, z_k) \in (\CC^n)^k$ such that the polydisks $\PD_{r_i}(z_i)$ have disjoint closures and are all contained in $\PD_s (0)$. The collections $\PD(r_1, \cdots, r_k \vert s)$ form an $\mathbb{R}_{> 0}$--colored operad in the category of complex manifolds under insertions of polydisks. 

Suppose that $\sF$ is a discretely translation-invariant prefactorization algebra on $\CC^n$ with values in $\DVS$. We may then identify $\sF(\PD_r (z)) \simeq \sF(\PD_r(z'))$ for any two $z, z' \in \CC^n$ using the isomorphisms (\ref{phi}), and denote the corresponding complex simply by $\sF_r$. For each $p \in \PD(r_1, \cdots, r_k \vert s )$, we have a multilinear map
\begin{equation} \label{mp}
m[p]: \sF_{r_1} \times \cdots \times \sF_{r_k} \mapsto \sF_{s}
\end{equation}

As explained in Section \ref{DVS}, we say that $m[p]$ \emph{depends smoothly on p} if $$m \in \DVS(\sF_{r_1}, \cdots, \sF_{r_k}, \bf{C}^{\infty}(\PD(r_1, \cdots, r_k \vert s), \cF_s)).$$ To formulate the definition of smooth translation-invariance, we will need the notion of a derivation of a prefactorization algebra. 

\begin{dfn}[\cite{CG}]
A degree k derivation of a prefactorization algebra $\sF$ is a collection
of maps $D_ U: \sF (U) \mapsto \sF (U)$ of cohomological degree $k$ for each open subset
$U \subset M$, with the property that for any finite collection $U_1, \cdots, U_n \subset V$ of disjoint
opens and elements $\alpha_i \in \sF(U_i)$, the following version of the Leibniz rule holds
$$
D_V m^{U_1, \cdots, U_n}_V (\alpha_1, \cdots, \alpha_n) = \sum_{i} (-1)^{k( \vert \alpha_1 \vert + \cdots + \vert \alpha_{i-1}) \vert } m^{U_1, \cdots, U_n}_V ( \alpha_1, \cdots, \alpha_{i-1}, D_{U_i} \alpha_i, \cdots, \alpha_n)
$$
\end{dfn}

The derivations of $\sF$ form a DGLA, with bracket $[D, D']_U = [D_U, D'U]$ and differential $d$ given by $dD_U = [d_U, D_U]$, where $d_U$ is the differential on $\sF(U)$. 

The notion of smoothly translation-invariant prefactorization algebra $\cF$ on $\CC^n$ can now be formulated as follows:

\begin{dfn}[\cite{CG}]
A {\em prefactorization algebra} $\cF$ on $\CC^n$ with values in $\DVS$ is (smoothly) translation-invariant if:
\begin{enumerate}
\item $\cF$ is discretely translation-invariant.
\item The maps (\ref{mp}) are smooth as functions of $p  \in \PD(r_1, \cdots, r_k \vert s ) $
\item $\cF$ carries an action of the complex Abelian Lie algebra $\CC^n$ by derivations compatible with differentiating $m[p]$. 
\end{enumerate}

We can further refine the notion of translation-invariance to consider the holomorphic structure. 
We say that $\sF$ is {\it holomorphically translation invariant}
if
\begin{itemize}
\item $\sF$ is smoothly translation invariant. 
\item There exist degree $-1$ derivations $\eta_i: \sF \to
  \sF$ such that 
  \begin{itemize}
   \item $[\dd, \eta_i] = \frac{\partial}{ \partial_{\Bar{z}_i}}$ (as derivations of
  $\sF$)
  \item  $[\eta_i, \eta_j] = [\eta_i, \frac{\partial}{ \partial_{\Bar{z}_j}}] = 0$
  \end{itemize} for $i=1, \cdots, n$, and where $\dd$ is the differential on $\cF$. 
\end{itemize}

This condition means that anti-holomorphic vector fields act homotopically trivially on $\sF$. 
\end{dfn}

As explained in \cite{CG}, if $\sF$ is a holomorphically translation-invariant prefactorization algebra, then upon passing to cohomology, the induced structure maps 
\begin{equation} \label{mphol}
m[p]: H^*(\sF_{r_1}) \times \cdots \times H^*(\sF_{r_k}) \mapsto H^*(\sF_{s})
\end{equation}
are holomorphic as functions of $p \in \PD(r_1, \cdots, r_k \vert s )$. In other words, $m$ can be viewed as a map
\begin{equation} \label{mp_hol}
m: H^*(\sF_{r_1}) \times \cdots \times H^*(\sF_{r_k}) \mapsto \on{Hol}(\PD(r_1, \cdots, r_k \vert s ), H^*(\sF_{s}))
\end{equation}
where $\on{Hol}$ denotes the space of holomorphic maps in  $\DVS$. 

\begin{eg}
For any Lie algebra $\fg$, the holomorphic factorization envelope $\U(\fg \otimes \Omega^{0,*}_{\CC^n})$ of Example \ref{gDolbeault} is holomorphically translation-invariant, with $\eta_i = \frac{d}{d(\dzbar_i)}$
\end{eg}

\section{Factorization algebras from holomorphic fibrations} \label{main_construction_section}

In this section, we describe our main construction of prefactorization algebras from locally trivial holomorphic fibrations. 

Our starting point is the following data:
\begin{itemize}
\item Complex manifolds $F, X$.
\item $(\fg, \langle, \rangle)$ a Lie algebra with an invariant bilinear form.
\item A locally trivial holomorphic fibration $\pi: E \rightarrow X$ with fiber $F$.
\end{itemize}

To this data we will construct a sheaf of $L_{\infty}$ algebras on the total space $E$ of the fibration. 
In turn, we obtain a prefactorization algebra upon taking its factorization envelope. 

\subsection{A sheaf of Lie algebras} \label{sec: cocycle}

Let $\Omega^{0,*}_E$ be the sheaf of dg vector spaces given by the Dolbeault complex of the complex manifold $E$, equipped with the $\delbar$ operator. 
For any Lie algebra $\fg$, we can define the sheaf of dg Lie algebras
\[
\fg_E = \fg \otimes \Omega^{0,*}_E
\]
which on an open set $U \subset E$ assigns $\fg \otimes \Omega^{0,*}(U)$. 
The differential is again given by the $\delbar$ operator, and the Lie bracket is defined by
\[
[X \otimes \alpha, Y \otimes \beta] = [X,Y] \otimes \alpha \wedge \beta 
\]
where $X, Y \in \fg$, $\alpha, \beta \in \Omega^{0,*}(U)$. 

Let $K_{E}$ be the double complex of sheaves
\[
\CC [2] \rightarrow \Omega^{0,*}_E [1] \overset{\partial}{\rightarrow } \Omega^{1,*}_E 
\]
where the first arrow is an inclusion, and the second is given by the holomorphic de Rham operator $\partial$. 
The complex $K_E$ is viewed as a double complex with $\delbar$ acting as vertical differential on $\Omega^{p,*}_E$, $p=0,1$. 
Let $\cK_{E} = \on{Tot}(K_{E})$ be the totalization of this double complex. 

We consider the complex of sheaves $\cK_E$ as a trivial dg module for the sheaf of dg Lie algebras $\fg_E$. 
We will construct a cocycle on $\fg_E$ with values in this trivial dg module. 
To this end, define the following maps of sheaves
\begin{align*}
\phi^{(1)} & : (\fg_{E})^{\otimes 2} \rightarrow  \Omega^{1,*}_E \\
\phi^{(1)} ((X \otimes \alpha) \otimes (Y \otimes \beta)) & = \frac{1}{2} \langle X, Y \rangle  (\alpha \wedge \del \beta - (-1)^{| \alpha |} \del \alpha \wedge \beta),
\end{align*}
and
$$ \phi^{(0)}: (\fg_{E})^{\otimes 3} \rightarrow \Omega^{0,*}_E [1] $$
$$ \phi^{(0)}( (X \otimes \alpha) \otimes(Y \otimes \beta) \otimes (Z \otimes \gamma)) = \frac{1}{2} \langle [X,Y], Z \rangle (\alpha \wedge \beta \wedge \gamma) $$
We view the sum $\phi = \phi^{(0)} + \phi^{(1)}$ as a cochain in the Chevalley-Eilenberg complex $$ \phi \in \CEcoh(\fg_{E}, \cK_{E}) $$ of total degree $2$. 
A short calculation shows the following:

\begin{lemma}\label{lem: cocycle}
$\phi$ defines a cocycle in $\mc C_{Lie}^*(\fg_{\pi}, \cK_{E}) $ of total degree $2$.
\end{lemma}

With the cocycle $\phi$ in hand, we can make the following definition of a complex computing the ``twisted Lie algebra homology" of $\fg_E$.
First, let
\[
\fghat_E := \fg_E \oplus \cK_{E} 
\] 
be the sheaf of dg vector spaces obtained by taking the direct sum of $\fg_E$ and $\cK_E$. 

\begin{dfn}
Define the sheaf of dg vector spaces
\[
\Ch(\fghat_{E}) := \left(\on{Sym}(\fghat_{E}[1]), d + d_{CE} + \phi\right)
\]
where 
\begin{itemize}
\item $d = \delbar + d_{\cK_{E}}$ is the differential on $\fghat_{E}$ (with the summands acting on $\fg_{E}, \cK_{E}$ respectively)
\item $d_{CE}$ is the Chevalley-Eilenberg differential of $\fg_{E}$
\item the linear map $\phi$ is extended to $\on{Sym}(\fghat_{E}[1])$ as a co-derivation. 
\end{itemize}
The fact that $( d + d_{CE} + \phi)^2 = 0$ follows from Lemma \ref{lem: cocycle}.
\end{dfn}

\begin{rmk}
We use the notation $\Ch(\fghat_{E})$ since this complex is the Chevalley-Eilenberg complex of a sheaf of $L_\infty$ algebras.
Indeed, one can think of $\phi$ as defining the structure of a sheaf of $L_\infty$-algebras on $\fghat_E$: the $\ell_1$ bracket equal to the sum of the differentials, the $\ell_2$ bracket is
\[
\ell_2(X \tensor \alpha, Y \tensor \beta) = [X \tensor \alpha, Y \tensor \beta] + \phi^{(1)} (X \tensor \alpha, Y \tensor \beta) 
\]
and the $\ell_3$ bracket
\[
\ell_3(X \tensor \alpha, Y \tensor \beta, Z \tensor \gamma) = \phi^{(0)} (X \tensor \alpha, Y \tensor \beta, Z \tensor \gamma) .
\]
All other brackets are zero.
\end{rmk}

\subsubsection{Relation to local cocycles}

There is a relationship to our construction and the theory ``twisted factorization envelopes" given in Section 4.4 of \cite{CG2} when the complex dimension $\dim_{\CC}(E) = 1$. 
There, the data one uses to twist is that of a local cocycle which lives in the local cohomology of a sheaf of Lie algebras. 
We don't recall the precise definition, but if $\cL$ is a sheaf of Lie algebras obtained from a bundle $L$, the local cohomology $\mc C_{loc}^*(\cL)$ is a sheaf of dg vector spaces which is witnessed as a subsheaf
\[
\mc C_{loc}^*(\cL) \subset \mc C_{Lie}^*(\cL_c)
\]
where $\cL_c$ denotes the cosheaf of compactly supported sections. 
The condition for a cochain in $\mc C_{Lie}^*(\cL_c)$ to be local is that it is given by integrating a ``Lagrangian density". 
Such a Lagrangian density is a differential form valued cochain which only depends on the $\infty$-jet of the sections of $L$, that is, it is given by a product of polydifferential operators. 

We have defined the cocycle $\phi$ as an element in $\mc C^*_{Lie}(\fg_E, \cK_E)$. 
The complex $\mc C^*_{Lie}(\fg_E, \cK_E)$ is neither a sheaf or a cosheaf, however, the object 
\[
\mc C^*_{Lie}(\fg_{E, c}, \cK_E)
\]
is a sheaf. 
Here, we restrict to cochains defined on compactly supported sections of $\fg_E$. 
The cocycle $\phi$ is a section of this sheaf, meaning it is compatible with the natural restriction maps. 

The cocycle $\phi$ is not just any section of this sheaf.
For any open $U \subset E$, it actually lies in the subcomplex
\[
\phi(U) \in \mc C^*_{Lie}(\fg_{E, c} (U), \cK_{E,c}(U)) \subset \mc C^*_{Lie}(\fg_{\pi, c}, \cK_E)(U) .
\]
In other words, $\phi$ preserves the condition of being compactly supported. 

Now, the cosheaf $\cK_{E,c}$ admits a natural integration map
\[
\int : \cK_{E,c} \to \underline{\C} [-1]
\]
where $\underline{\C} [-1]$ is the constant cosheaf concentrated in degree $+1$.  
(Integration is only nonzero on the $\Omega^{1,1}_c$, which accounts for the shift above)
Thus, for every $U \subset E$, we obtain a cocycle
\[
\phi(U) \in \mc C^*_{Lie}(\fg_{E, c} (U)) .
\]
The cocycle $\phi$ is clearly built from differential operators, which implies that $\phi(U)$ is actually a element of the local cochain complex
\[
\phi(U) \in \mc C_{loc}^*(\fg_E) (U) .
\]
In conclusion, upon integration, we see that $\phi$ determines a degree one cocycle in the local cohomology of the sheaf of Lie algebras $\fg_E$. 

In higher dimensions, there is a similar relationship to local functionals which hence determine one dimensional central extensions of Kac-Moody type algebras in any dimension. 
This class of cocycles is studied in detail in the context of ``higher dimensional" Kac-Moody algebras in \cite{GWkm}. 

\subsection{The prefactorization algebra $\sF_{\fg,\pi}$}

We proceed to construct a prefactorization algebra on $X$ --- the base of the fibration $\pi: E \rightarrow X$. For an open subset $U \subset X$, define the dg vector space
\begin{equation} \label{fghat_def}
\fghat^c_{\pi}(U) = \Gamma_c (U,  \pi_* (\fghat_{E}))
\end{equation}
In other words, $\fghat^c_{\pi}$ is the cosheaf of compactly supported sections of $\pi_*( \fghat_{E})$. 

\begin{rmk}
Though the assignment $U \mapsto \fghat^c_{\pi} (U)$ is a cosheaf of dg vector spaces, it is just a {\em pre}cosheaf of $L_{\infty}$ algebras on $X$ (with $L_\infty$ structure defined by the cocycle $\phi$ in the previous subsection). 
This subtle issue arises since direct sum is not the categorical coproduct in the category of Lie algebras, but it will play no essential role for us. 
\end{rmk}

\begin{dfn}
Define the precosheaf
\begin{equation} \label{FEX}
\sF_{\fg, \pi} := \Ch ( \fghat^c_{\pi} )
\end{equation}
as the Chevalley-Eilenberg complex for Lie algebra homology for the cosheaf $\fghat^c_\pi$.
For each open $U \subset X$ this cosheaf assigns the complex
\begin{equation} 
\sF_{\fg, \pi}(U) = \Ch (\fghat^c_{\pi}(U))
\end{equation}
\end{dfn}

\begin{rmk}
The cosheaf $\fghat^c_{\pi}$ is equipped with a $\DVS$ structure as in Example \ref{direct_image_DVS}, and therefore so is $\sF_{\fg, \pi}$, being constructed from algebraic tensor product.
\end{rmk}

\begin{prop}\label{prop: f}
$\sF_{\fg, \pi}$ has the structure of a pre-factorization algebra on $X$ valued in $dg-\DVS$. When $X= \CC^n$,  $\sF_{\fg, \pi}$ is holomorphically translation-invariant. 
\end{prop}

\begin{proof}
Note that if $\pi: E \to X $ is a locally trivial fibration and $W$ is a smooth vector bundle on $E$, then $\pi_* \wt{W}$ is a fine sheaf. 
The smooth translation-invariance of $\sF_{\fg, \pi}$ is established just as in the example of the free scalar field in Section 4.8 of \cite{CG}. Finally, $\eta_i= \frac{d}{d(d\zbar_i))}$ are degree $(-1)$ derivations satisfying the conditions in the definition of holomorphic translation-invariance. 
\end{proof}

\subsection{The prefactorization algebra $\sG_{\fg,\pi}$ }

In this section we discuss some prefactorization algebras closely related to $\sF_{\fg,\pi}$, which are both more convenient from a computational standpoint and more closely related to the class of toroidal algebras we have introduced previously. 
While the definition of $\sF_{\fg,\pi}$ is reasonably simple, explicit calculations of $H^*(\sF_{\fg,\pi} (U))$ for an open subset $U \subset X$ require the $\delbar$-cohomology of the complex $ \Gamma_c (U,  \pi_* (\fghat_{E}))$ in (\ref{fghat_def}). This complex involves forms with compact support along the base $X$ and arbitrary support along the fiber $F$, and its $\delbar$-cohomology even when $E$ is a trivial fibration is a certain completion of $H^{0,*}_c (U) \otimes H^{0,*} (F) $ whose explicit description involves non-trivial analytic issues, due to the failure of naive Kunneth-type theorems for Dolbeault cohomology. 


Let us first suppose that $E = X \times F$ is a trivial fibration. We have a map of cosheaves
\[
\Omega^{p,q}_{X,c} \otimes \Gamma(F, \Omega^{{p', q'}}_F) \rightarrow (\pi_* \Omega^{p+p', q+q'}_E)_c
\]
where the subscript $c$ denotes sections with compact support. Explicitly, for an open subset $U \subset X$, this map is just the wedge product
\begin{align*}
\Omega^{p,q}_{X,c}(U) \otimes \Gamma(F, \Omega^{p',q'}_F) & \rightarrow \Gamma_c(U,  \pi_* \Omega^{p+p', q+q'}_E) \\
\alpha \otimes \beta &\rightarrow \alpha \wedge \beta
\end{align*}
It is injective provided all three factors are non-zero. 

\begin{dfn}
Define a sub cosheaf $\ffg$ of $\fghat^{c}_{\pi}$ by
\[
\ffg := \fg \otimes \Omega^{0,*}_{X,c} \otimes \Gamma(F, \Omega^{0,*}_F) \oplus \mc{K}^{\#}_{\pi}
\]
where
\[
\mc{K}^{\#}_{\pi} := \on{Tot} (  \Omega^{0,*}_{X,c} \otimes \Gamma(F, \Omega^{0,*}_F) \overset{\partial}{\rightarrow} \Omega^{1,*}_{X,c} \otimes  \Gamma(F, \Omega^{0,*}_F) \oplus \Omega^{0,*}_{X, c} \otimes \Gamma(F, \Omega^{1,0}_F)) .
\]
Here, $\delbar$ acts "vertically" within each term. 
\end{dfn}

\begin{rmk}
The cosheaf $\ffg$ may be equipped with a $dg-\DVS$ structure as in Example \ref{tensor_DVS}. 
\end{rmk}

The $L_{\infty}$ structure on $\fghat^c_{\pi}$ induces one on the sub-complex $\ffg$. 
The thing to observe is that the cocycle $\phi$ restricts to one on this subcomplex.
The advantage of $\ffg$ lies in the fact that it's constructed from ordinary (algebraic) tensor products of complexes whose cohomology is easy to describe. 

\begin{dfn}
Define the precosheaf $\sG_{\fg, \pi} $ on $X$ by
\[
\sG_{\fg, \pi} :=  \Ch ( \ffg )
\]
The same argument for $\sF_{\fg,\pi}$ in Proposition \ref{prop: f} implies that $\sG_{\fg, \pi}$ has the structure of a prefactorization algebra on $X$. 
\end{dfn}


The arguments of the previous section show:

\begin{prop}
$\sG_{\fg, \pi}$ has the structure of a pre-factorization algebra on $X$. When $X= \CC^n, n\geq 1$,  $\sG_{\fg, \pi}$ is holomorphically translation-invariant. 
\end{prop}

When $X = \CC^n$, there is a natural map $\sG_{\fg, \pi} \to \sF_{\fg, \pi}$. 

\begin{lemma}
Suppose $X = \CC^n$ and $\pi$ is a trivial holomorphic fibration. 
Then, there exists a map of factorization algebras on $\CC^n$:
\[
\sG_{\fg, \pi} \to \sF_{\fg, \pi}
\]
\end{lemma}
\begin{proof}
We have just constructed a map of DGLA's $\ffg \to \fghat^{c}_{\pi}$, which induces a map  $ \sG_{\fg, \pi} \to \sF_{\fg, \pi} $ upon taking the factorization enveloping algebra.
\end{proof}

For a general locally trivial holomorphic fibration $\pi : E \to X$, we can construct a map of factorization algebras $\sG_{\fg, \pi} \to \sF_{\fg, \pi}$ locally on $X$. 
Indeed, for any $x \in X$ there is a local coordinate $(z_1,\ldots, z_d)$ near $x$ trivializing the fibration. 
This map of factorization algebras depends on both the choice of local coordinate and trivialization.  

\subsubsection{An algebraic variant} \label{sec: algvar}

When the fiber $F$ is a smooth complex affine variety and $X=\CC^n$, we may further refine $\sG_{\fg,\pi}$ to obtain a prefactorization algebra $\sG^{alg}_{\fg,\pi}$ with stronger finiteness properties, by considering the algebraic rather than analytic cohomology of $\mc{O}_F$. This variation will be important in the next section, when we make contact with vertex algebras. Let $\mc{O}^{alg}_F$ denote the sheaf of algebraic regular functions on $F$, and $\Omega^{1, alg}$ the sheaf of Kahler differentials. Since $F$ is affine, hence Stein, we have
\begin{align*}
H^{0}(F, \mc{O}^{alg}_F) \subset H^{0}(F, \mc{O}_F) & \hookrightarrow (\Omega^{0,*}_F, \delbar) \\
H^{0}(F, \Omega^{1, alg}_F ) \subset H^{0}(F, \Omega^{1}_F) & \hookrightarrow (\Omega^{1,*}_F, \delbar )
\end{align*}

\begin{dfn}
We define 
\[
\ffga := \fg  \otimes \Omega^{0,*}_{X,c} \otimes  H^0(F, \mc{O}^{alg}_F)  \oplus \mc{K}^{\#, alg}_{\pi}
\]
where
\[
\mc{K}^{\#, alg}_{\pi} := \on{Tot} (  \Omega^{0,*}_{X,c} \otimes H^0(F, \mc{O}^{alg}_F)  \overset{\partial}{\rightarrow} \Omega^{1,*}_{X,c} \otimes  H^0(F, \mc{O}^{alg}_F) \oplus \Omega^{0,*}_{X, c} \otimes H^{0}(F, \Omega^{1, alg}_F )) .
\]
The totalization is respect to the horizontal $\partial$-operator and the the vertical $\delbar$-operator acting on  $\Omega^{p,*}_{X, c}$. 
\end{dfn}

\begin{rmk}
Again, $\ffga$ may be equipped with the $\DVS$ structure of Example \ref{tensor_DVS}, yielding a prefactorization algebra in $dg-\DVS$.  
\end{rmk}

We can now define the main object of study for us.

\begin{dfn}
The $n$-dimensional {\em toroidal prefactorization algebra} associated with the holomorphic fibration $F \to E \xto{\pi} \CC^n$ is the prefactorization algebra
\[
\sG^{alg}_{\fg, \pi} :=  \Ch \left( \ffga \right)
\]
with the structure maps induced from those of $\sG_{\fg, \pi}$. 
\end{dfn}

\begin{prop}
Suppose that $F$ is a smooth complex affine variety and $X=\CC^n$. Then
$\sG^{alg}_{\fg,\pi}$ has the structure of a holomorphically translation-invariant pre-factorization algebra valued in $dg-\DVS$. 
\end{prop}

The reasoning at the end of the previous section shows that for a general locally trivial fibration $\pi: E \to X$, a choice of point $x \in X$, local coordinates $z_1, \cdots, z_n$ on $x \in U \subset X$, and trivialization of $E \vert_U$, yields prefactorization algebra maps
$$ \sG^{alg}_{\fg,\pi} \to  \sG_{\fg, \pi} \to \sF_{\fg, \pi} \vert_U $$

\begin{rmk}
We remark that in this paper we are primarily concerned with the case in which the fiber $F$ is a smooth complex affine variety. 
There is another interesting case in which we take the fibers to be compact. 
One can still study the pushforward $\pi_*(\Hat{\fg}_E)$ as a sheaf of $L_\infty$ algebras, and its factorization envelope. 
Unlike the affine case, this pushforward has interesting cohomology in the fiber direction, and moreover the factorization enveloping algebra is equipped with the analytic Gauss-Manin connection. 
The case of a trivial fibration has been observed in Section 4.3 of \cite{GWkm}. 
\end{rmk}

%
%

\section{$X=\CC$ and vertex algebras} \label{one_dim_section}

In \cite{CG}, it is shown that prefactorization algebras on $X=\CC$ which are holomorphically translation-invariant and $S^1$-equivariant for the natural action by rotations are closely related to vertex algebras. More precisely, given such a prefactorization algebra $\cF$, the vector space
\[
V(\cF)=\bigoplus_{l}  H^*(\cF^{(l)}(\CC))
\]
equal to the direct sum of $S^1$-eigenspaces in the cohomology $H^*(\cF(\CC))$ has a vertex algebra structure.   
We begin by reviewing this correspondence following \cite{CG}, and then apply it to the case of the one-dimensional toroidal prefactorization algebra $\cG^{alg}_{\fg, \pi}$, where $\pi: \CC \times F \mapsto \CC$ is the trivial fibration on $\CC$ with fiber a smooth complex affine variety $F$.  
We show that resulting vertex algebra is isomorphic to $V(\fghat_{R})$ where $R=H^*(F, \mc{O}^{alg}_F)$ from Section \ref{VR}. 
As a special case, when $F=(\CC^{*})^k$, we recover a toroidal vertex algebra.

\subsection{Prefactorization algebras on $\CC$ and vertex algebras}

We review here the correspondence between prefactorization algebras on $\CC$ and vertex algebras established in \cite{CG}, where we refer the reader for details. Recall that $S^1$ acts on $\CC$ by rotations via $z \mapsto \exp(i \theta) z$. Suppose that $\cF$ is a prefactorization algebra on $\CC$ that is holomorphically translation-invariant and $S^1$-equivariant. Let $\sF(r) := \sF(D(0,r))$ be the complex assigned by $\sF$ to a disk of radius $r$ (we allow here $r=\infty$, in which case $D(0, \infty) = \CC$), and $\sF^{(l)}(r) \subset \sF(r)$ be the $l$th eigenspace for the $S^1$-action. The following theorem from \cite{CG} establishes a bridge between prefactorization and vertex algebras:

\begin{theorem}[Theorem 5.2.2.1 \cite{CG}] \label{fv} Let $\sF$ be a unital $S^1$-equivariant holomorphically translation invariant  prefactorization algebra on $\CC$. Suppose
\begin{itemize}
\item The action of $S^1$ on $\sF(r)$ extends smoothly to an action of the algebra of distributions on $S^1$. 
\item For $r < r'$ the map 
\[
\sF^{(l)}(r) \to \sF^{(l)}(r')
\]
is a quasi-isomorphism.
\item The cohomology ${\rm H}^*(\sF^{(l)}(r))$ vanishes for $l \gg 0$.
\item For each $l$ and $r > 0$ we require that ${\rm H}^*(\sF^{(l)}(r))$ is isomorphic to a countable sequential colimit of finite dimensional vector spaces. 
\end{itemize}
Then $V(\sF) := \bigoplus_l {\rm H}^*(\sF^{(l)}(r))$ (which is independent of $r$ by assumption) has the structure of a vertex algebra.
\end{theorem}

We briefly sketch how the vertex algebra structure on $V(\sF)$ can be extracted from the prefactorization structure on $\sF$.

\begin{itemize}
\item Polydisks in one dimension are simply disks, and we denote $ \PD(r_1, \cdots, r_k \vert s )$ by $\D(r_1, \cdots, r_k \vert s )$. If $r'_i < r_i$, we obtain an inclusion 
\begin{equation} \label{disk_inclusion}
\D(r_1, \cdots, r_k \vert s) \subset  \D(r'_1, \cdots, r'_k \vert s )
\end{equation}
In the limit $\lim_{r_i \to 0}$, these spaces approach $\on{Conf}_k$, the configuration space of $k$ distinct points in $\CC$.  
\item The structure maps \ref{mp_hol} are compatible with the maps $\cF_{r'_i} \mapsto \cF_{r_i}$ and the inclusions \ref{disk_inclusion}, and one may take $\lim_{r_i \to 0}$, $s=\infty$, obtaining maps
\begin{equation} \label{limit_structure}
m: (\lim_{r \to 0} H^{*}(\cF_r) )^{\otimes k}  \to \on{Hol}( \on{Conf}_k, H^{*}(\cF(\CC))
\end{equation}
\item We set $k=2$, and fix one of the points to be the origin. There is a natural map $V(\cF) \mapsto \lim_{r \to 0} H^{*}(\cF_r) $, as well as projections 
$H^{*}(\cF(\CC)) \to H^{*}(\cF^{(l)}(\CC))$. Pre and post-composing by these in \ref{limit_structure}, yields a map
\begin{equation} \label{vop}
\ol{m_{0,z}} : V(\cF) \otimes V(\cF) \to \prod_{l} \on{Hol}(\CC^{\times}, V(\cF)_l) 
\end{equation}
where $V(\cF)_l = H^{*}(\cF^{(l)}(\CC))$.
 Laurent expanding $\overline{m_{0,z}}$ we obtain
\[
\overline{m_{0,z}} : V(\cF) \otimes V(\cF) \to \prod_l V(\cF)_l [[z,z^{-1}]]
\] 
whose image can be shown to lie in $V(\cF)((z))$. The vertex operator can now be defined by
\begin{align*}
Y: V(\cF) & \to \on{End}(V(\cF))[[z,z^{-1}]] \\
Y(v,z) v'  &= m_{0,z}(v', v) 
\end{align*}
\item Holomorphic translation invariance yields an action of $\partial_z$
\[
\partial_z : \sF^{(l)}(r) \to \sF^{(l-1)}(r) .
\]
which descends to $H^*(\cF^{l}(r))$. This induces the translation operator $T : V(\cF) \to V(\cF)$. 
\item The vacuum vector is obtained from the unit in $\cF(\emptyset)$. 

\end{itemize}

\subsection{The main theorem}

Our goal in this section is to prove the following theorem

\begin{theorem} \label{main_theorem}
Let $F$ be a smooth complex affine variety, and $\pi: \CC \times F \to \CC$ the trivial fibration with fiber $F$. Then
\begin{enumerate}
\item The toroidal prefactorization algebra $\sG^{alg}_{\fg,\pi}$ satisfies the hypotheses of Theorem \ref{fv}
\item  The vertex algebra $V(\sG^{alg}_{\fg,\pi})$ is isomorphic to the toroidal vertex algebra $V(\fghat_{\R})$, with $R = H^0(F, \mc{O}^{alg}_F)[t,t^{-1}]$ defined in Section \ref{VR}.
\end{enumerate}
\end{theorem}

Throughout this section, $R$ will denote the algebra $H^0(F, \mc{O}^{alg}_F)[t,t^{-1}]$. We will denote $H^{0}(F, \mc{O}^{alg}_F)$ simply by $\CC[F]$, so $R=\CC[F][t,t^{-1}]$. Recall that 
\begin{align*}
\fghat_R &= \fg \otimes \CC[F][t,t^{-1}] \oplus \Omega^1_{\CC[F][t,t^{-1}]} / d (\CC[F][t,t^{-1}])  \\
	&= \fg \otimes \CC[F][t,t^{-1}] \oplus  \frac{ \CC[t,t^{-1}]\otimes\Omega^1_{\CC[F]} \oplus \CC[F] \otimes \Omega^1_{\CC[t,t^{-1}]}}{ \langle t^k du + k t^{k-1} u dt  \rangle}
\end{align*}

\subsubsection{Recollections on Dolbeault cohomology and preliminary computations}

In this section we recall some facts regarding ordinary and compactly supported Dolbeault cohomology and apply these to compute $H^*( \sG^{alg}_{\fg,\pi}(U)) $ over opens $U \subset \CC$. These results will be used in proving Theorem \ref{main_theorem}. 

Stein manifolds are complex analytic analogues of smooth affine varieties over $\CC$ \cite{Stein}.
In particular, $\CC^n$ are the basic examples of Stein manifolds.
In addition, all open subsets $U \subset \CC$ are Stein. 
We recall the following classic result pertaining to Stein manifolds:

\begin{theorem}[Cartan's Theorem B] \label{CartanB}
Let $X$ be a Stein manifold. Then
\[
H^k(\Omega^{p,*}(X), \delbar) =\begin{cases} 
      0 & k \neq 0 \\
      \Omega^p_X & k = 0 
   \end{cases}
\]
where $\Omega^p_X$ denotes the space of holomorphic $p$-forms on $Z$. 
\end{theorem}

\begin{rmk}
When $n >1$ the open subset $\CC^n \setminus 0 \subset \CC^n$ is {\em not} Stein, since it has higher cohomology. 
\end{rmk}

On a complex manifold $X$ of dimension $n$, Serre duality implies that there is a non-degenerate pairing between ordinary and compactly supported forms
\begin{align*}
\Omega^{p,q}_{X,c} & \otimes \Omega^{n-p, n-q}_X  \to \CC \\
\alpha & \otimes \beta \to \int_X \alpha \wedge \beta .
\end{align*}
Thus, compactly supported differential forms yield continuous linear functionals on differential forms. 
At the level Dolbeault cohomology, one obtains the following corollary to Threorem \ref{CartanB} noted by Serre (\cite{Serre}): 

\begin{cor} \label{Serre_corollary}
Let $X$ be a Stein manifold. Then
\[
H^k(\Omega^{p,*}_c (X), \delbar) =\begin{cases} 
      0 & k \neq dim(X) \\
      (\Omega^{n-p}_X (X))^{\vee} & k =n = dim(X)
   \end{cases}
\]
where $(\Omega^{n-p}_X (X) )^{\vee}$ denotes the continuous dual to the space of holomorphic $n-p$ forms with respect to the Frechet topology. 
\end{cor}

We now turn to our situation.
Here, we are looking at the holomorphic fibration $F \to E \xto{\pi} X$ where $X=\CC$, $F$ is a complex affine variety, and $\pi$ is the trivial fibration. 
The cosheaf $\ffga$ on $\CC$ defined in Section \ref{sec: algvar} has the form
\[
\ffga := \fg \otimes \CC[F] \otimes  \Omega^{0,*}_{X,c} \oplus \mc{K}^{\#, alg}_{\pi}
\]
where $\mc{K}^{\#, alg}_{\pi}$ is the total complex of the following double complex
\[\begin{tikzcd}
\Omega^{0,1}_c \otimes \CC[F] \arrow{r}{\partial + d }  & \Omega^{1,1}_c \otimes \CC[F] \oplus \Omega^{0,1}_c \otimes \Omega^1_{\CC[F]} \\
\Omega^{0,0}_c \otimes \CC[F] \arrow{r}{\partial + d} \arrow[swap]{u}{\delbar} & \Omega^{1,0}_c \otimes \CC[F] \oplus \Omega^{0,0}_c \otimes \Omega^1_{\CC[F]}  \arrow{u}{\delbar} .
\end{tikzcd}
\]
Here, $\Omega^{p,q}_c$ denotes the cosheaf of compactly supported forms on $\CC$, and $\Omega^1_{\CC[F]}$ is the space of \emph{algebraic} $1$-forms (i.e. Kahler differentials) on $F$. 
Recall, from this cosheaf we have defined the factorization algebra
\[
\sG^{alg}_{\fg,\pi} = \Ch( \ffga) = \on{Sym}(\ffga[1], \wt{d})
\]
where the differential $\wt{d}$ may be decomposed as  $\wt{d} = d_1 + d_2 + d_3$, with $$d_i : \on{Sym}^i (\ffga[1]) \to \ffga[1]$$ of cohomological degree $1$ defined by:
\begin{itemize}
\item $ d_1 = \delbar + d_{\mc{K}^{\#, alg}_{\pi}}$,
the linear differential operators defining the underlying cochain complexes of Dolbeault forms;
\item $d_2 = d_{CE,\fg}$,
the Chevalley-Eilenberg differential induced from the Lie bracket on $\fg$;
\item $d_3 = \phi$,
where $\phi$ is the cocycle defined in Section \ref{sec: cocycle}, which extends to $\Ch(-)$ by the rule that it is a coderivation. 
\end{itemize}

The complex $ \Ch( \ffga) $ has an increasing filtration by symmetric degree, leading to a spectral sequence whose $E_0$ page is $\on{Sym}(\ffga[1], d_1)$, as $d_2$ and $d_3$ lower symmetric degree. We have
\[
H^*(\on{Sym}(\ffga[1], d_1)) = \on{Sym}(H^*(\ffga[1], d_1))
\]
Now $(\ffga, d_1)$ is the direct sum of the complexes $(\fg \otimes \CC[F] \otimes \Omega^{0,*}_c, \delbar)$ and $\mc{K}^{\#, alg}_{\pi}$. Applying Theorem \ref{CartanB} on an open Stein subset $U \subset \CC$, the cohomology of the first is $$\fg \otimes \CC[F] \otimes (\Omega^1 (U))^{\vee}.$$ Similarly, by first computing the $\delbar$ cohomology in $\mc{K}^{\#, alg}_{\pi}$, we have
\[
H^*(\mc{K}^{\#, alg}_{\pi}(U) ) =  \on{Coker} \left( (\Omega^1_X (U))^{\vee} \otimes \CC[F] ) \overset{1 \otimes \partial + \partial^{\vee} \otimes 1}{\rightarrow} (\Omega^{1}_X (U))^{\vee} \otimes \Omega^1_{\CC[F]} \oplus (\mc{O} (U))^{\vee} \otimes \CC[F] \right)
\]
where $\partial^{\vee}$ denotes the transpose of $\partial: \Omega^{n-1}_X(U) \mapsto \Omega^n_X (U)$.
We have the following:

\begin{lemma}
Let $U \subset \CC$ be an open subset. Then 
\begin{align*}
H^*(\ffga(U), d_1) &= \fg \otimes (\Omega^1_X (U))^{\vee} \otimes \CC[F]  \bigoplus \\
& \on{Coker} \left( (\Omega^1_X (U))^{\vee} \otimes \CC[F] ) \overset{1 \otimes \partial + \partial^{\vee} \otimes 1}{\rightarrow} (\Omega^{1}_X (U))^{\vee} \otimes \Omega^1_{\CC[F]} \oplus (\mc{O} (U))^{\vee} \otimes \CC[F] \right)
\end{align*}
where $\partial^{\vee}$ denotes the transpose of $\partial: \mc{O}_X(U) \mapsto \Omega^1_X (U)$.
\end{lemma}

It follows that $H^*(\ffga[1], d_1)$ is concentrated in cohomological degree $0$. Since $\wt{d}$ has cohomological degree $1$, this means that the spectral sequence computing $H^*(\sG^{alg}_{\fg,\pi}(U))$ collapses at $E_1$, and we have an isomorphism of vector spaces
\[
H^*(\sG^{alg}_{\fg,\pi}(U)) \simeq \on{Sym}( H^*(\ffga(U)[1], d_1) )
\]

Suppose now that $U \subset \CC$ is a disk or annulus centered at $0$. We have a non-degenerate residue pairing
\begin{align*}
\Omega^1 (\CC^*) \otimes \mc{O}(\CC^*) & \to \CC \\
\omega \otimes f & \to \on{Res}_{z=0} f \omega
\end{align*}
which yields maps  $\Omega^{1} (\CC^*) \to (\mc{O}(U))^{\vee}$ and $\mc{O}(\CC^*) \to (\Omega^{1} (U))^{\vee}$. 

\begin{lemma} \label{coh_fa}
Let $z$ be a global coordinate on $\CC$, and $U=D(0,r)$ a disk.  There is an isomorphism of vector spaces
\begin{equation}
V(\sG^{alg}_{\fg,\pi}) \simeq \on{Sym}(\fghat_{S}/ \fghat^{+}_{S} ) \simeq U(\fghat_{S}) \otimes_{U(\fghat^{+}_{S}} \CC
\end{equation}
where $S=\CC[F][z,z^{-1}]$.
\end{lemma}

\begin{proof}
We introduce the vector spaces
\begin{align*} 
S^+ &= \CC[F][z] \\
S^{-} &= \CC[F] \otimes z^{-1} \CC[z^{-1}] \\
\Omega^1_{S^{+}} &= \Omega^1_{\CC[F]} \otimes \CC[z] \oplus \CC[z]dz \otimes \CC[F] \\
\Omega^1_{S^{-}} &= \Omega^1_{\CC[F]} \otimes z^{-1} \CC[z^{-1}] \oplus z^{-1}\CC[z^{-1}]dz \otimes \CC[F] 
\end{align*}
We have $S=S^+ \oplus S^{-}$ and $\Omega^1_{S} = \Omega^1_{S^{+}} \oplus \Omega^1_{S^{-}} $ as vector spaces, and these decompositions are moreover compatible with the differential, in the sense that $d(S^{\pm}) \in \Omega^1_{S^{\pm}}$. Hence
\[
\Omega^1_S/ dS \simeq \Omega^1_{S^+}/d S^+ \oplus \Omega^1_{S^-}/d S^-
\]
which implies that as vector spaces
\[
\fghat_S/\fghat^{+}_S \simeq \fg \otimes S^{-} \oplus  \Omega^1_{S^-}/d S^-
\]
We note that the $S^1$ action on $\CC$ extends naturally to $\ffga(D(0,r))$ and  $\sG^{alg}_{\fg,\pi}(D(0,r))$, and that the differential $\wt{d}$ on the latter is $S^1$-equivariant. 
Using the residue pairing to identify $\CC[z,z^{1}]$ with a subspace of $(\Omega^1(U))^{\vee}$ and $\CC[z,z^{-1}] dz$ with a subspace of $(\mc{O}(D(0,r)))^{\vee}$, we obtain that 
\[
H^*(\ffga(D(0,r)), d_1)^{(l)} \simeq \fg \otimes \CC[F] \otimes \lbrace z^{l} \rbrace \oplus \left( \Omega^1_{\CC[F]} \otimes  \lbrace z^{l} \rbrace  \oplus \CC[F] \otimes \lbrace z^{l-1} dz \rbrace \right) / \on{im} (\wt{d}) 
\]
when $l \geq 0$ and $0$ otherwise. Therefore, as vector spaces
\[
V(\sG^{alg}_{\fg,\pi}) = \on{Sym} \left(\bigoplus_{l \geq 0} H^*(\ffga(D(0,r)), d_1)^{(l)} \right) = \on{Sym} ( \fg \otimes S^{-} \oplus  \Omega^1_{S^-}/d S^-  ) = \on{Sym}(\fghat_{S}/ \fghat^{+}_{S} )
\]

\end{proof}

\subsubsection{Verifying the hypotheses of Theorem \ref{fv}}

We proceed to verify the hypotheses of Theorem \ref{fv}, establishing part $(1)$ of Theorem \ref{main_theorem} above.
\begin{itemize}
 \item The first hypothesis is verified as in Section {5.4.2} of \cite{CG}
 \item The second and third hypotheses follow from Lemma \ref{coh_fa}, from which it follows in particular that $H^*( (\sG^{alg}_{\fg,\pi} (D(0,r)) )^{(l)})$ is non-zero only if $l \leq 0$
 \item The last hypothesis requires some attention. We note that it is necessary to ensure that the structure map $m_{z,0}$ possesses a Laurent expansion. By Lemma \ref{coh_fa}
$H^*( (\sG^{alg}_{\fg,\pi} (D(0,r)) )^{(l)})$ may be identified with the elements of weight $l$ in 
\[
\on{Sym} \left( \fg \otimes S^{-} \oplus  \Omega^1_{S^-}/d S^- 	\right)
\]
where $z$ and $dz$ have $S^1$ - weight $1$. We begin by showing that $\CC[F]$ and $\Omega^1_{\CC[F]}$ are naturally a sequential colimit of finite-dimensional vector spaces. This can be done as follows. Embed $F \subset \mathbb{A}^N = \on{Spec} \CC[x_1, \cdots, x_N]$. This induces an increasing filtration $F^k\CC[F]$, $k \geq 0$, where $F^k\CC[F]$ is spanned by the images of polynomials of degree $\leq k$ in $x_1, \cdots, x_N$. $\CC[F]$ and by the same reasoning $\Omega^1_{\CC[F]}$ can therefore be expressed as a countable union of finite-dimensional vector spaces. This induces a filtration on $\ffga$ compatible with the $\DVS$ structure, which in turn induces one on $H^*( (\sG^{alg}_{\fg,\pi} (D(0,r)) )^{(l)})$. 
\end{itemize}
\subsubsection{Constructing the isomorphism}

We proceed to prove part $(2)$ of Theorem \ref{main_theorem}. The proof is a variation on the approach taken in \cite{BWvir} with respect to the Virasoro factorization algebra, and involves three main steps:
\begin{enumerate}
\item Showing that $V(\sG^{alg}_{\fg,\pi})$ has the structure of a $\fghat_R$-module. 
\item Showing that $V(\sG^{alg}_{\fg,\pi}) \simeq V(\fghat_R)$ as $\fghat_R$-modules.
\item Checking that the vertex algebra structures agree by using the reconstruction theorem \ref{rec_thm}. 
\end{enumerate}

Let $\rho: \CC^{\times} \to \RR_{>0}$ be the map $\rho(z) = z \ol{z} = \vert z \vert^2$. 
The universal enveloping algebra $\mc{U}(\fghat_R)$ defines a locally constant prefactorization algebra on $\RR_{> 0}$ which we denote $\mc{A} \mc{U} (\fghat_R)$. 

\begin{lemma} \label{homo_prop}
There is a homomorphism $\Phi: \mc{A} \mc{U} (\fghat_R) \to \rho_* H^*( \sG^{alg}_{\fg,\pi}) $ of prefactorization algebras on $\RR_{> 0}$. 
\end{lemma}

\begin{proof}
\begin{itemize}
\item It is shown in \cite{CG} that a map of prefactorization algebras on $\RR_{>0}$ is determined by the maps $\Phi_I$ on connected open intervals. 
For each open interval $I \subset \RR_{>0}$, $A_I = \rho^{-1}(I)$ is an annulus. We choose for each such a bump function $f_{I}: A_I \to R$ having the properties 
\begin{itemize}
\item $f$ is a function of $r^2=z \ol{z}$ only.
\item $f \geq 0$ and $f$ is supported in $A_I$.
\item $\int_A f dz d\ol{z} =1$. 
\end{itemize}
The map $\Phi_I$ is uniquely determined by where it sends the generators of $\fghat_R$. 
We define $\Phi_I$ on these linear generators by the assignments:
\begin{align*}
\Phi_I( J \otimes u t^k ) &= - [ J \otimes u z^{k+1} f_I  d\zbar ] \\
\Phi_I( t^k \omega ) &= [ z^{k+1}  f_{I} \omega \wedge d \zbar ] \\
\Phi_I( t^k  u dt ) &= [ u z^{k+1} f_{I} dz d \zbar] 
\end{align*}
where $J \in \fg, u \in \CC[F], \omega \in \Omega^1_{\CC[F]}$, and $[-] \in  H^*( \sG^{alg}_{\fg,\pi}(A_I))$ denotes the $\delbar$-cohomology class of the closed differential form. 
The elements on the right are clearly closed for the differential $\wt{d}$, and the corresponding $\wt{d}$-cohomology classes are easily seen to be independent of the choice of the function $f$. 

\item We first check that $\Phi_I$ is well-defined, which amounts to verifying that $$\Phi_I (d (u t^k)  ) = \Phi_I (t^k du + k t^{k-1} u dt) =  0 \in H^*( \sG^{alg}_{\fg,\pi}(A_I)) $$ for each $u \in \CC[F], k \in \mathbb{Z}$. We have
\[
\Phi_I (t^k du + k t^{k-1} u dt) = - [z^{k+1} f_{I} d \zbar du ] - [k z^k u f_{I} dz d\zbar] 
\]
We note that this cohomology class lies in the image of the natural map 
\begin{equation} \label{proj1}
 \Omega^{1,1}_c(A_I)   \otimes \CC[F] \oplus \Omega^{0,1}_c (A_I)   \otimes \Omega^1_{\CC[F]} \subset  (\ffga( A_I))^{cl} \to  H^*( \sG^{alg}_{\fg,\pi}(A_I))
\end{equation}
where $(\ffga(A_I))^{cl}$ denotes the subspace of sections closed for the differential $d_1$, and therefore $\wt{d}$. The projection (\ref{proj1}) may be factored by first taking the quotient by $\on{im} (\delbar) $, in other words as
\[
 \Omega^{1,1}_c(A_I)   \otimes \CC[F] \oplus \Omega^{0,1}_c (A_I)  \otimes \Omega^1_{\CC[F]} \to ( \Omega^{1}_X (A_I))^{\vee} \otimes \Omega^1_{\CC[F]} \oplus (\mc{O} (A_I))^{\vee} \otimes \CC[F] \to  H^*( \sG^{alg}_{\fg,\pi}(A_I))
\]
We calculate  
\begin{align*}
-[z^{k+1}& f_{I} du d \zbar ] - [k z^k u f_{I} dz d\zbar] + [d_{\mc{K}^{\#, alg}_{\pi}}( z^{k+1} u f_I d \zbar) ] \\
&= -[z^{k+1} f_{I} du d \zbar ]  -[k z^k u f_{I} dz d\zbar] +
 [ z^{k+1} f_I du d \zbar + (k+1) z^k  u f_I dz d \zbar  + u z^{k+1} \partial f \wedge d \zbar ] \\
 &= + [u (z^k f_I dz d \zbar + z^{k+1} \partial f \wedge d \zbar )]
\end{align*}
It therefore suffices to show that $[z^k f_I dz d \zbar + z^{k+1} \partial_z f_I \wedge d \zbar ] = 0 \in (\mc{O}_X (A_I))^{\vee}$, or equivalently, that 
\[
\int_{A_I} z^m \wedge (z^k f_I dz d \zbar + z^{k+1} \partial f \wedge d \zbar ) = 0 \; \; \forall m \in \mathbb{Z}
\]
This follows from the identities
\begin{equation} \label{identities}
\int_{A_I} z^a f_I dz d\zbar = \delta_{a,0} \; \; \; ,  \int_U z^b \partial f \wedge d \zbar = - \delta_{b,1}
\end{equation}
for $a,b \in \Z$, which follows from integrating by parts. 

\item Consider three disjoint open intervals $I_1, I_2, I_3 \subset \RR_{>0}$, such that $I_{i+1}$ is located to the right of $I_i$, all contained in a larger interval $I$. Their inverse images under $\rho$ correspond to three nested annuli $A_{I_i}$ inside a larger annulus $A_I$. We have structure maps
\[
\bullet_{i,i+1} : \rho_* H^*( \sG^{alg}_{\fg,\pi})(I_i) \otimes \rho_* H^*( \sG^{alg}_{\fg,\pi})(I_{i+1}) \to \rho_* H^*( \sG^{alg}_{\fg,\pi})(I) \; \; i=1, 2
\]
To show that $\Phi$ is a prefactorization algebra homomorphism, we have to check that for $X, Y \in \fghat_{R}$, 
\[
\Phi_{I_1} (X) \bullet_{1,2} \Phi_{I_2} (Y) - \Phi_{I_2}(Y) \bullet_{2,3} \Phi_{I_3} (X) = \Phi_I ([X,Y]) 
\]
Let 
\[
F_m(z,\zbar) = z^m \int^{z \zbar}_0 (f_{I_1}(s) - f_{I_3}(s)) ds
\]
Then on $A_{I_2}$, $F_m = z^m$, and moreover,
\begin{align*}
\delbar F_m (z, \zbar) & = z^{m} \delbar (  \int^{z \zbar}_0 (f_{I_1}(s) - f_{I_3}(s)) ds ) \\ 
			          & = z^m \frac{\partial (z \zbar)}{\partial \zbar} \frac{\partial}{\partial (z \zbar)} ( \int^{z \zbar}_0 (f_{I_1}(s) - f_{I_3}(s)) ds) d \zbar \\
			          & = z^{m+1} ( f_{I_1} (z \zbar) - f_3 ( z \zbar) ) d \zbar
			          \end{align*}

Let $J_1, J_2 \in \fg, \; u, v \in \CC[F]$. Then
\begin{align*}
\Phi_{I_1} & (J_1 u t^k ) \bullet_{1,2} \Phi_{I_2} (J_2 v t^l) - \Phi_{I_2}(J_2 v t^l ) \bullet_{2,3} \Phi_{I_3} (J_1 u t^k) - \Phi_{I_2} ([J_1 u t^k, J_2 v t^l] ) \\
  &=  \Phi_{I_1}  (J_1 u t^k ) \bullet_{1,2} \Phi_{I_2} (J_2 v t^l) - \Phi_{I_2}(J_2 v t^l ) \bullet_{2,3} \Phi_{I_3} (J_1 u t^k)  \\ & \; \; \; \; \; - \Phi_{I_2} \left( [J_1, J_2] uv t^{k+l} + \frac{1}{2} \langle J_1, J_2 \rangle( u t^k d(v t^l) - v t^l d(u t^k) ) \right) \\
& = \left( [J_1 u z^{k+1} f_{I_1} d\zbar] \cdot  [J_2 v z^{l+1} f_{I_2} d\zbar] -  [J_2 v z^{l+1} f_{I_2} d\zbar] \cdot  [J_1 u z^{k+1} f_{I_3} d\zbar] \right) + \\
& \; \; \; \; + [ [J_1, J_2] uv z^{k+l+1} f_{I_2} d \zbar ] + \frac{1}{2} \langle J_1, J_2 \rangle [z^{k+l+1} f_{I_2} (u dv - v du) d \zbar + (l-k) u v z^{k+l} f_{I_2} dz d\zbar] \\
\end{align*}

We also have 
\begin{align*}
\wt{d}  & \left(  [J_1 u F_k] \cdot [J_2 v z^{l+1} f_{I_2} d \zbar]  \right) \\
&= \left( [J_1 u z^{k+1} f_{I_1} d\zbar] \cdot  [J_2 v z^{l+1} f_{I_2} d\zbar] -  [J_2 v z^{l+1} f_{I_2} d\zbar] \cdot  [J_1 u z^{k+1} f_{I_3} d\zbar] \right) + [ [J_1, J_2] uv z^{k+l+1} f_{I_2} d \zbar ] \\
& \; \; \; \; + \frac{1}{2} \langle J_1, J_2 \rangle [ u F_k \partial(v z^{l+1} f_{I_2} d \zbar  ) - \partial( F_k u) \wedge (z^{l+1} v f_{I_2} d \zbar )] \\
&= \left( [J_1 u z^{k+1} f_{I_1} d\zbar] \cdot  [J_2 v z^{l+1} f_{I_2} d\zbar] -  [J_2 v z^{l+1} f_{I_2} d\zbar] \cdot  [J_1 u z^{k+1} f_{I_3} d\zbar] \right) + [ [J_1, J_2] uv z^{k+l+1} f_{I_2} d \zbar ] \\
& \; \; \; \; + \frac{1}{2} \langle J_1, J_2 \rangle [ z^{k+l+1}(u dv - v du)f_{I_2} d \zbar  + uv ( (l-k+1) f_{I_{2}}z^{k+l} dz d \zbar - z^{k+l+1} \partial f_{I_2} \wedge d \zbar)] \\
\end{align*}
where we have used the fact that over the support of $f_{I_2}$, $F_k = z^k$. Using the identities (\ref{identities}), we obtain
\[
[ (l-k+1) z^{k+l} f_{I_{2}} dz d \zbar - z^{k+l+1} \partial f_{I_2} \wedge d \zbar] = [(l-k) z^{k+l} f_{I_2} dz d \zbar ].
\]
It follows that 
\[
\Phi_{I_1}  (J_1 u t^k ) \bullet_{1,2} \Phi_{I_2} (J_2 v t^l) - \Phi_{I_2}(J_2 v t^l ) \bullet_{2,3} \Phi_{I_3} (J_1 u t^k) - \Phi_{I_2} ([J_1 u t^k, J_2 v t^l] ) = 0 \in  \rho_*H^*( \sG^{alg}_{\fg,\pi})(I)
\]
proving the lemma. 
\end{itemize}
\end{proof}

The homomorphism $\Phi$ of Proposition (\ref{homo_prop}) equips $V(\sG^{alg}_{\fg,\pi})$ with the structure of a $\fghat_R$-module. Let us fix $0 < r < r' < R$ We have the following commutative diagram:
\[
\begin{tikzcd}
H^*( \sG^{alg}_{\fg,\pi}(D(0,r))) \otimes H^*( \sG^{alg}_{\fg,\pi} (A(r',R)))  \arrow{r}{m}  & H^*( \sG^{alg}_{\fg,\pi} (D(0, R)) \\
V(\sG^{alg}_{\fg,\pi}) \otimes \arrow{u}{\iota \otimes \Phi} \mc{A} \mc{U} (\fghat_R) \arrow[r, dotted] & V(\sG^{alg}_{\fg,\pi}) \arrow{u}{\iota}
\end{tikzcd}
\]
where $\iota$ denotes the inclusion of $V(\sG^{alg}_{\fg,\pi}) \subset H^*( \sG^{alg}_{\fg,\pi}(D(0,r)))$ (for any $r$), and $m$ is the prefactorization structure map. As explained in \cite{CG}, the existence of the dotted arrow (i.e. the fact that the $\mc{U}(\fghat_R)$-action preserves the subspace $V(\sG^{alg}_{\fg,\pi})$) follows from the fact that the structure map $m$ is $S^1$-equivariant. 

In concrete terms, the action of $X \in \fghat_R$ on $v \in V(\sG^{alg}_{\fg,\pi})$ is given as follows: we may represent $v$ by a closed chain $\wt{v} \in \Ch(\ffga(D(0,r))$ - then $X\cdot v$ is represented by $\Phi_{(r', R)}(X)\cdot \wt{v}$. 

\begin{lemma}
There is an isomorphism of $\fghat_R$-modules $$  \eta: V(\fghat_R) \to V(\sG^{alg}_{\fg,\pi}) $$ which sends $\vac \in V(\fghat_R)$ to $1 \in V(\sG^{alg}_{\fg,\pi})$. 
\end{lemma}

\begin{proof}
Let $h(z,\zbar) = \int^{z \zbar}_0 f(s) ds$. By the chain rule, we have that 
\[
\delbar (z^n h(z,\zbar) )= z^{n+1} f(z \zbar) d \zbar
\]
Thus in $H^*( \sG^{alg}_{\fg,\pi} (D(0, R))$, we have for $k \geq 0$:
\begin{align*}
\Phi_{(r', R)}( J u t^k) & = [J u z^{k+1} f(z \zbar) d \zbar ] = \wt{d} ( J u z^k h(z, \zbar))   \\
\Phi_{(r', R)}( t^k u dv) & = [ z^{k+1} f(z \zbar)  u d \zbar dv] = \wt{d} (  z^k h(z, \zbar) u dv) \\
\Phi_{(r', R)}( u t^k dt) & = [ u z^{k+1} f(z \zbar) d \zbar dz ] = \wt{d} ( u z^k h(z, \zbar) dz)
\end{align*}
In other words, if $X \in \fghat^+_R$, then $\Phi_{(r', R)}(X) = 0 \in H^*( \sG^{alg}_{\fg,\pi} (D(0, R))$. This shows that the vector $1 \in V(\sG^{alg}_{\fg,\pi}) $ is annihilated by $\fghat^+_R$. It follows that there exists a unique map of $\fghat_R$-modules $\eta: V(\fghat_R) \to V(\sG^{alg}_{\fg,\pi})$ sending $\vac \to 1$. It remains to show this is an isomorphism, which can be done as in \cite{CG, BWvir} for the affine and Virasoro algebra, so we will be brief. Both $V(\fghat_R)$ and $V(\sG^{alg}_{\fg,\pi})$ have the structure of filtered $\mc{U}(\fghat_R)$-modules, where in each case the filtration is induced by symmetric degree. It is straightforward to verify that  $\eta$ induces an isomorphism at the level of associated graded modules, proving the result. 

\end{proof}

To complete the proof of Theorem \ref{main_theorem}, we check that $\eta$ induces an isomorphism of vertex algebras.
Suppose that $z \in A((r', R))$. Recall that the operation $$ Y: V(\sG^{alg}_{\fg,\pi}) \otimes V(\sG^{alg}_{\fg,\pi}) \to V(\sG^{alg}_{\fg,\pi})((z)) $$ is induced from the diagram
\[
\begin{tikzcd}
V(\sG^{alg}_{\fg,\pi}) \otimes V(\sG^{alg}_{\fg,\pi}) \arrow{d}{\iota \otimes \iota_z}  & \\
H^*( \sG^{alg}_{\fg,\pi}(D(z,\epsilon))) \otimes H^*( \sG^{alg}_{\fg,\pi}(D(0, r))) \arrow{r}{m_{z,0}} & H^*( \sG^{alg}_{\fg,\pi}(D(0,R)))
\end{tikzcd}
\]
as the Laurent expansion of the map $m_{z,0} \circ \iota \otimes \iota_z$. By the Reconstruction Theorem \ref{rec_thm}, it suffices to show that the generating field assignments agree, that is we need to verify that for $v \in V(\sG^{alg}_{\fg,\pi})$, 
\begin{align*}
m_{z,0} ( \iota_{z} (\eta(Jut^{-1} \cdot \vac)), \iota(v)) &= \sum_{n \in \mathbb{Z}} (\Phi(Jut^{n}) \cdot v) z^{-n-1} \\
m_{z,0} ( \iota_z( \eta(ut^{-1}dt \cdot \vac)), \iota(v)) &= \sum_{n \in \mathbb{Z}} (\Phi(ut^{n-1} dt) \cdot v) z^{-n} \\
m_{z,0} ( \iota_z(\eta(t^{-1} \omega \cdot \vac)), \iota(v)) &= \sum_{n \in \mathbb{Z}} (\Phi(t^{n} \omega) \cdot v) z^{-n-1} \\
\end{align*}
$\iota_{z} (\eta(Jut^{-1} \cdot \vac))$ may be identified with the element $ J u \psi_z \in   \fg \otimes \CC[F] \otimes (\Omega^1 (D(z,\epsilon)))^{\vee}$, where 
$\psi_z \in (\Omega^1 (D(z,\epsilon)))^{\vee}$ is defined by
\[
\psi_z (h(w) dw) = \frac{1}{2 \pi i } \oint_{C(z, \delta)} \frac{h(w) dw}{w-z}
\]
By the residue theorem, for $h(w) dw \in \Omega^1(A(r, R))$, we may switch contours, to write
\begin{align*}
 \oint_{C(z, \delta)} \frac{h(w) dw}{w-z} &=  \oint_{C(0, R-\delta)} \frac{h(w) dw}{w-z} -  \oint_{C(0, r' + \delta)} \frac{h(w) dw}{w-z} \\
 							&= \sum_{n \geq 0} (\oint_{C(0, R-\delta)} w^{-n-1} h(w) dw) z^{n}  +  \sum_{n < 0} (\oint_{C(0, r'+\delta)} w^{-n-1} h(w) dw) z^{n} \\
\end{align*}
 where in the second line we have expanded $\frac{1}{w-z}$ into a geometric series in the domains $|w| > |z|$ and $|w| < |z|$ respectively. Using the fact that
 \[
 Res_0 h(w) w^{-n-1} dw = \int_{A(r', R)} h(w) w^{-n}  f_{(r, R)} dw d \overline{w} 
 \]
and $\Phi(J u t^{-n-1}) \cdot v = [ J u z^{-n} f_{(r', R)} d \zbar] \cdot v$ we obtain the first identity. Similarly, we may identify $ \iota_z( \eta(ut^{-1}dt \cdot \vac)) $ with the element $u \xi_z \in \CC[F] \otimes (\mc{O}(D(z,\epsilon)))^{\vee}$, where
\[
\xi_z (h(w)) = h(z) =  \frac{1}{2 \pi i } \oint_{C(z, \delta)} \frac{h(w) dw}{w-z}
\]
and $\iota_z(\eta(t^{-1} \omega \cdot \vac))$ with $\omega \psi_z \in \Omega^1_{\CC[F]} \otimes   (\Omega^1 (D(z,\epsilon)))^{\vee}$. Expanding these in contour integrals centered at $0$, and identifying the coefficients with appropriate elements in the image of $\Phi$ as above proves the remaining two identities.  

\newpage

\bibliographystyle{alpha}
\bibliography{toroidal_bib}

\address{\tiny DEPARTMENT OF MATHEMATICS AND STATISTICS, BOSTON UNIVERSITY, 111 CUMMINGTON MALL, BOSTON} \\
\indent \footnotesize{\email{szczesny@math.bu.edu}}

\address{\tiny DEPARTMENT OF MATHEMATICS AND STATISTICS, BOSTON UNIVERSITY, 111 CUMMINGTON MALL, BOSTON} \\
\indent \footnotesize{\email{jackwalt@bu.edu}}

\address{\tiny DEPARTMENT OF MATHEMATICS, NORTHEASTERN UNIVERSITY, LAKE HALL, BOSTON}
\indent \footnotesize{\email{br.williams@northeastern.edu}}

\end{document}